\documentclass[reqno,english]{amsart}%
\usepackage{amsfonts,amsmath,latexsym,verbatim,amscd,mathrsfs,color,array}
\usepackage[colorlinks=true]{hyperref}
\usepackage{amsmath,amssymb,amsthm,amsfonts,graphicx,color}
\usepackage{amssymb}
\usepackage{pdfsync}
\usepackage{epstopdf}
\usepackage{cite}
\usepackage{graphicx}
\usepackage{amsmath}
\usepackage{amsfonts}%
\providecommand{\U}[1]{\protect\rule{.1in}{.1in}}
\numberwithin{equation}{section}

\newtheorem{theorem}{Theorem}[section]
\newtheorem{corollary}{Corollary}[section]
\newtheorem{lemma}{Lemma}[section]
\newtheorem{proposition}{Proposition}[section]
\newtheorem{remark}{Remark}[section]
\newtheorem{definition}{Definition}[section]

\numberwithin{equation}{section}

\newcommand{\bbr}{\mathbb{R}}

\newcommand{\bbn}{\mathbb{N}}
\newcommand{\ve}{\varepsilon}
\newcommand{\bd}{\begin{definition}}
\newcommand{\ed}{\end{definition}}

\newcommand{\br}{\begin{remark}}
\newcommand{\er}{\end{remark}}

\newcommand{\be}{\begin{equation}}
\newcommand{\ee}{\end{equation}}

\newcommand{\bc}{\begin{corollary}}
\newcommand{\ec}{\end{corollary}}
\begin{document}

\title[Normalized solutions]{Normalized solutions for Schr\"odinger equations with critical Sobolev exponent and mixed nonlinearities}

\author[J. Wei]{Juncheng Wei}
\address{\noindent Department of Mathematics, University of British Columbia,
Vancouver, B.C., Canada, V6T 1Z2}
\email{jcwei@math.ubc.ca}

\author[Y.Wu]{Yuanze Wu}
\address{\noindent  School of Mathematics, China
University of Mining and Technology, Xuzhou, 221116, P.R. China }
\email{wuyz850306@cumt.edu.cn}

\begin{abstract}
In this paper, we consider the following nonlinear Schr\"{o}dinger equations with mixed nonlinearities:
\begin{eqnarray*}
\left\{\aligned
&-\Delta u=\lambda u+\mu |u|^{q-2}u+|u|^{2^*-2}u\quad\text{in }\mathbb{R}^N,\\
&u\in H^1(\bbr^N),\quad\int_{\bbr^N}u^2=a^2,
\endaligned\right.
\end{eqnarray*}
where $N\geq3$, $\mu>0$, $\lambda\in\bbr$ and $2<q<2^*$.  We prove in this paper
\begin{enumerate}
\item[$(1)$]\quad Existence of solutions of mountain-pass type for $N=3$ and $2<q<2+\frac{4}{N} $.
\item[$(2)$]\quad Existence and nonexistence of ground states for $2+\frac{4}{N}\leq q<2^*$ with $\mu>0$ large.
\item[$(3)$]\quad Precisely asymptotic behaviors of ground states and mountain-pass solutions as $\mu\to0$ and $\mu$ goes to its upper bound.
\end{enumerate}
Our studies answer some questions proposed by Soave in \cite[Remarks~1.1, 1.2 and 8.1]{S20}.

\vspace{3mm} \noindent{\bf Keywords:} Normalized solution; Ground state; Variational method; Critical nonlinearity; Mixed nonlinearities.

\vspace{3mm}\noindent {\bf AMS} Subject Classification 2010: 35B09; 35B33; 35B40; 35J20.%

\end{abstract}

\date{}
\maketitle

\section{Introduction}
In this paper, we consider the following nonlinear scalar field equation:
\begin{eqnarray}\label{eqn03}
\left\{\aligned
&-\Delta u=\lambda u+\mu |u|^{q-2}u+|u|^{2^*-2}u\quad\text{in }\mathbb{R}^N,\\
&u\in H^1(\bbr^N),
\endaligned\right.
\end{eqnarray}
where $N\geq3$, $\mu>0$, $\lambda\in\bbr$ and $2<q<2^*=\frac{2N}{N-2}$.

\vskip0.12in

\eqref{eqn03} is a special case of the following model,
\begin{eqnarray}\label{eqn02}
\left\{\aligned
&-\Delta u=\lambda u+f(u)\quad\text{in }\mathbb{R}^N,\\
&u\in H^1(\bbr^N),
\endaligned\right.
\end{eqnarray}
which is related to finding the stationary waves of nonlinear Schr\"odinger equations:
\begin{eqnarray}\label{eqn01}
i\psi_t+\Delta \psi+g(|\psi|^2)\psi=0\quad\text{in }\mathbb{R}^N.
\end{eqnarray}
Indeed, a stationary wave of \eqref{eqn01} is of the form $\psi(t,x)=e^{i\lambda t}u(x)$ where $\lambda\in\bbr$ is a constant and $u(x)$ is a time-independent function, then it is well-known that $\psi$ is a solution of \eqref{eqn01} if and only if $u$ is a solution of \eqref{eqn02} with $f(u)=g(|u|^2)u$.  As pointed out in \cite{S20,S201}, in general, the function $u$ is complex valued and thus, \eqref{eqn02} can be regarded as a complex valued system coupled by the nonlinearities $f(u)=g(|u|^2)u$.

\vskip0.12in

As pointed out in \cite{GS14,JL19}, the studies on \eqref{eqn02} can be traced back to the semi-classical papers \cite{BL83,BL831,L84,L841,S77}.  In these studies, there are two different methods to study \eqref{eqn02}.  The first one is to fix the number $\lambda<0$ in \eqref{eqn02} and restrict the unknown $u$ to be real valued.  In this case, \eqref{eqn02} is a single equation and under some mild assumptions on $f(u)$, the solutions of \eqref{eqn02} are critical points of the functional
\begin{eqnarray*}
\mathcal{J}(u)=\frac{1}{2}\int_{\bbr^N}(|\nabla u|^2-\lambda |u|^2)dx-\int_{\bbr^N}F(u)dx
\end{eqnarray*}
in the usual Sobolev space $H^1(\bbr^N)$, where $F(u)=\int_0^u f(t)dt$.  In this case, particular attention is devoted to least action solutions, namely solutions minimizing $\mathcal{J}(u)$ among all non-trivial solutions.  We also refer the readers to \cite{AIIK19,AIKN12,AIKN13,ASM12} and the references therein for the recent results on the special case \eqref{eqn03} in this direction.  Another one is to fix the $L^2$ norm of the unknown $u$, that is, to find solutions of \eqref{eqn02} with prescribed mass.  In this case, \eqref{eqn02} is always rewritten as follows
\begin{eqnarray}\label{eqn04}
\left\{\aligned
&-\Delta u=\lambda u+f(u)\quad\text{in }\mathbb{R}^N,\\
&u\in H^1(\bbr^N),\quad\int_{\bbr^N}u^2=a^2,
\endaligned\right.
\end{eqnarray}
where $\int_{\bbr^N}u^2=a^2$ is the prescribed mass, and in this case, $\lambda\in\bbr$ is a part of the unknown which appears in \eqref{eqn04} as a Lagrange multiplier.  In particular, \eqref{eqn03} is rewritten as
\begin{eqnarray}\label{eq0001}
\left\{\aligned
&-\Delta u=\lambda u+\mu |u|^{q-2}u+|u|^{2^*-2}u\quad\text{in }\mathbb{R}^N,\\
&u\in H^1(\bbr^N),\quad\int_{\bbr^N}u^2=a^2.
\endaligned\right.
\end{eqnarray}
Similar to the first case, under some mild assumptions on $f(u)$, the solutions of \eqref{eqn04} are critical points of the functional
\begin{eqnarray*}
\mathcal{F}_\mu(u)&=&\int_{\mathbb{R}^N}\big(\frac{1}{2}|\nabla u|^2-F(u)\big)dx
\end{eqnarray*}
on the smooth manifold
\begin{eqnarray*}
\mathcal{S}_a=\{u\in H^1(\bbr^N)\mid \|u\|_2^2=a^2\}.
\end{eqnarray*}
In particular, the solutions of \eqref{eq0001} are critical points of the $C^2$-functional
\begin{eqnarray*}
\mathcal{E}_\mu(u)&=&\frac{1}{2}\|\nabla u\|_2^2-\frac{\mu}{q}\|u\|_q^q-\frac{1}{2^*}\|u\|_{2^*}^{2^*}
\end{eqnarray*}
on $\mathcal{S}_a$, where we denote the usual norm in $L^p(\bbr^N)$ by $\|\cdot\|_p$.  In this case, solutions of \eqref{eqn04} are always called normalized solutions which are particularly relevant for the nonlinear Schr\"odinger equation~\eqref{eqn01} since the mass is preserved along the time
evolution in \eqref{eqn01}.  Thus, normalized solutions of \eqref{eqn04} seems to be particularly meaningful from the physical viewpoint, moreover, these solutions often offer a good insight of the dynamical properties of the stationary solutions for the nonlinear Schr\"odinger equation~\eqref{eqn01}, such as stability or instability (cf. \cite{BC81,CL82}).
In this case, particular attention is also devoted to least action solutions which are also called ground states for normalized solutions, namely solutions minimizing $\mathcal{F}_\mu(u)$ among all non-trivial solutions.  The studies on normalized solutions of \eqref{eqn04} is a hot topic in the community of nonlinear PDEs nowadays, thus, it is impossible for us to provide a complete references.  We just refer the readers to \cite{AW19,BS17,BV12,BJ16,BCGJ19,BES06,FM01,J97,JLW15,JL19,JL191,NTV14,PV17,L08,S17,S20,S201,S82} and the references therein.  In these references, we would like to highlight \cite{S20,S201,JL19} to the readers for their detail introductions and references on normalized solutions of \eqref{eqn04} and new directions on the study of normalized solutions of autonomous problems.  We also would like to point out \cite{GLW17,GLWZ19,GLWZ191,GS14,GWZ18,GZZ16,GZ20} and the references therein for the studies on normalized solutions of problems with trapping potentials.

\vskip0.12in

It is well-known that the number $p=2+\frac{4}{N}$ plays an important role in studying normalized solutions which is called the $L^2$ critical exponent or mass critical exponent in the literature.
Since $2^*=\frac{2N}{N-2}>2+\frac{4}{N}$, the nonlinearity of \eqref{eq0001} grows faster than $|u|^{p-2}u$ at infinity and thus, it is well-known that $\mathcal{E}_\mu(u)$ is unbounded from below on $\mathcal{S}_a$, which makes one to find new constraints to prove the existence of ground states of $\mathcal{E}_\mu(u)$ on $\mathcal{S}_a$.  The new constraint, which is introduced by Bartsch and Soave in \cite{BS17} for the general problem~\eqref{eqn04} and is widely used nowadays in studying normalized solutions, is the following $L^2$-Pohozaev manifold:
\begin{eqnarray*}
\mathcal{P}_{a,\mu}=\{u\in\mathcal{S}_a\mid \|\nabla u\|_2^2=\mu \gamma_q\|u\|_q^q+\|u\|_{2^*}^{2^*}\},
\end{eqnarray*}
where
\begin{eqnarray}\label{eq0008}
\gamma_q=\frac{N(q-2)}{2q}.
\end{eqnarray}
By the Pohozaev identity of \eqref{eq0001}, $\mathcal{P}_{a,\mu}$ contains all nontrivial solutions of \eqref{eq0001}, thus, we have the following definition of ground states of \eqref{eq0001}.
\begin{definition}
We say $(u_0,\lambda_0)$ is a ground state of \eqref{eq0001} if $u_0$ is a critical point of $\mathcal{E}_\mu|_{\mathcal{S}_a}(u)$ with $\mathcal{E}_\mu|_{\mathcal{S}_a}(u_0)=\inf_{u\in\mathcal{P}_{a,\mu}}\mathcal{E}_\mu(u)$.
\end{definition}

\vskip0.12in

The $L^2$-Pohozaev manifold $\mathcal{P}_{a,\mu}$ is quite related to the fibering maps
\begin{eqnarray*}
\Psi(u,s)=\frac{e^{2s}}{2}\|\nabla u\|_2^2-\frac{\mu e^{q\gamma_qs}}{q}\|u\|_q^q-\frac{e^{2^*s}}{2^*}\|u\|_{2^*}^{2^*},
\end{eqnarray*}
which is introduced by Jeanjean in \cite{J97} for the general problem~\eqref{eqn04} and is well studied by Soave in \cite{S20}.  According to the fibering maps $\Psi(u,s)$,
$\mathcal{P}_{a,\mu}$ can be naturally divided into the following three parts:
\begin{eqnarray*}
\mathcal{P}_+^{a,\mu}=\{u\in\mathcal{S}_a\mid 2\|\nabla u\|_2^2>\mu q\gamma^2_q\|u\|_q^q+2^*\|u\|_{2^*}^{2^*}\},\\
\mathcal{P}_0^{a,\mu}=\{u\in\mathcal{S}_a\mid 2\|\nabla u\|_2^2=\mu q\gamma^2_q\|u\|_q^q+2^*\|u\|_{2^*}^{2^*}\},\\
\mathcal{P}_-^{a,\mu}=\{u\in\mathcal{S}_a\mid 2\|\nabla u\|_2^2<\mu q\gamma^2_q\|u\|_q^q+2^*\|u\|_{2^*}^{2^*}\}.
\end{eqnarray*}
Let
\begin{eqnarray}\label{eq0045}
m_{a,\mu}^\pm=\inf_{u\in\mathcal{P}_\pm^{a,\mu}}\mathcal{E}_\mu(u),
\end{eqnarray}
then Soave proved the following results in \cite[Theorems~1.1 and 1.4]{S20}:
\begin{enumerate}
\item[$(1)$]\quad For $2<q<2+\frac{4}{N}$, there exists $\alpha_{N,q}>0$ such that if $\mu a^{q-q\gamma_q}<\alpha_{N,q}$ then
$m_{a,\mu}^+=\inf_{u\in\mathcal{P}_+^{a,\mu}}\mathcal{E}_\mu(u)=\inf_{u\in\mathcal{P}_{a,\mu}}\mathcal{E}_\mu(u)<0$
and it can be attained by some $u_{a,\mu,+}$ which is real valued, positive, radially symmetric and radially decreasing.  Moreover, \eqref{eq0001} has a ground state $(u_{a,\mu,+}, \lambda_{a,\mu,+})$ with $\lambda_{a,\mu,+}<0$, and $m_{a,\mu}^+\to0$ and $\|\nabla u_{a,\mu,+}\|_2\to0$ as $\mu\to0$.
\item[$(2)$]\quad For $2+\frac{4}{N}\leq q<2^*$, there exists $\alpha_{N,q}>0$ such that if $\mu a^{q-q\gamma_q}<\alpha_{N,q}$ then
$m_{a,\mu}^-=\inf_{u\in\mathcal{P}_-^{a,\mu}}\mathcal{E}_\mu(u)=\inf_{u\in\mathcal{P}_{a,\mu}}\mathcal{E}_\mu(u)\in(0, \frac{1}{N}S^{\frac{N}{2}})$
and it can be attained by some $u_{a,\mu,-}$ which is real valued, positive, radially symmetric and radially decreasing, where $S$ is the optimal constant in the Sobolev embedding, that is,
\begin{eqnarray}\label{eq0048}
\|u\|_{2^*}^2\leq S^{-1}\|\nabla u\|_2^2\quad\text{for all }u\in D^{1,2}(\bbr^N).
\end{eqnarray}
Moreover, \eqref{eq0001} has a ground state $(u_{a,\mu,-}, \lambda_{a,\mu,-})$ with $\lambda_{a,\mu,-}<0$, and $m_{a,\mu}^-\to\frac{1}{N}S^{\frac{N}{2}}$ and $\|\nabla u_{a,\mu,-}\|_2\to S^{\frac{N}{2}}$ as $\mu\to0$.
\end{enumerate}

\vskip0.12in

In the $L^2$-subcritical case $2<q<2+\frac{4}{N}$, since $\mathcal{E}_\mu(u)|_{\mathcal{S}_a}$ is unbounded from below, it could be naturally to expect that $\mathcal{E}_\mu(u)|_{\mathcal{S}_a}$ has a second critical point of mountain-pass type, which is also positive, real valued and radially symmetric.  This natural expectation has been pointed out by Soave in \cite[Remark~1.1]{S20} which can be summarized to be the following question:
\begin{enumerate}
\item[$(Q_1)$]\quad {\bf Does $\mathcal{E}_\mu(u)|_{\mathcal{S}_a}$ has a critical point of mountain-pass type in the $L^2$-subcritical case $2<q<2+\frac{4}{N}$?}
\end{enumerate}

\begin{remark}
In preparing this paper, we notice that in the very recent work \cite{JL201}, the question~$(Q_1)$ has been solved for $N\geq4$.  Thus, it only need to consider the case $N=3$ for the question~$(Q_1)$.
\end{remark}

\vskip0.12in

Besides, since Soave only considered the case that $\mu a^{q-q\gamma_q}>0$ small in \cite[Theorem~1.1]{S20}, it is also natural to ask what will happen if $\mu>0$ and $\mu a^{q-q\gamma_q}>0$ is large.  This natural question has been proposed by Soave in \cite{S20} as an open problem, which can be summarized to be the following one:
\begin{enumerate}
\item[$(Q_2)$]\quad {\bf Does $\mathcal{E}_\mu(u)|_{\mathcal{S}_a}$ have a ground state if $\mu>0$ and $\mu a^{q-q\gamma_q}>0$ large?}
\end{enumerate}
In \cite{S20}, Soave conjectures that the answer of $(Q_2)$ is {\it negative} in general.

\vskip0.12in

Finally, in these results, the asymptotic behavior is only for $\|\nabla u_{a,\mu,-}\|_2$ in the cases of $2+\frac{4}{N}\leq q<2^*$.  Thus, it is also natural to ask if it is possible to characterize the asymptotic behavior of $u_{a,\mu,-}$, and not only of $\|\nabla u_{a,\mu,-}\|_2$.  In \cite[Remark~8.1]{S20}, Soave pointed out that in dimensions $N=3,4$, it could be proved that $\|\nabla u_{a,\mu,-}\|_2\to S^{\frac{N}{2}}$, but $u_{a,\mu,-}\rightharpoonup0$ in $H^1$ while, in dimensions $H\geq5$, both $u_{a,\mu,-}\rightharpoonup0$ and $u_{a,\mu,-}\rightharpoonup \widetilde{u}\not=0$ could happen.  He also {\it conjectures} that the weak limit of $\{u_{a,\mu,-}\}$ will be the Aubin-Talanti babbles in the higher dimensions $N\geq5$.  Soave's conjecture can be slightly generalized to the following question:
\begin{enumerate}
\item[$(Q_3)$]\quad {\bf Can we capture the precisely asymptotic behavior of $u_{a,\mu,-}$ as $\mu\to0$?}
\end{enumerate}

In this paper, we are interested in these questions and we shall give some answers to them, which will give more information on the ground states of \eqref{eq0001}.   Our first result, which is devoted to the existence and nonexistence of ground states, can be stated as follows.
\begin{theorem}\label{thm0001}
Let $N\geq3$, $2<q<2^*$ and $a,\mu>0$.
\begin{enumerate}
\item[$(1)$]\quad If $N=3$ and $2<q<2+\frac{4}{N}$, then for $\mu a^{q-q\gamma_q}<\alpha_{N,q}$, $m_{a,\mu}^-$ can be attained by some $u_{a, \mu,-}$ which is real valued, positive, radially symmetric and radially decreasing, and thus, \eqref{eq0001} has a second solution $u_{a, \mu,-}$ with some $\lambda_{a, \mu,-}<0$.
\item[$(2)$]\quad If $q=2+\frac{4}{N}$, then $m_{a,\mu}^-$ can not be attained for $\mu a^{q-q\gamma_q}\geq\alpha_{N,q}$ and thus, \eqref{eq0001} has no ground states for $\mu a^{q-q\gamma_q}\geq\alpha_{N,q}$.
\item[$(3)$]\quad If $2+\frac{4}{N}<q<2^*$, then for all $\mu>0$,
and $m_{a,\mu}^-$ can be attained by some $u_{a, \mu,-}$ which is real valued, positive, radially symmetric and radially decreasing, and thus, \eqref{eq0001} has a ground state $u_{a, \mu,-}$ with some $\lambda_{a,\mu,-}<0$ for all $\mu>0$.
\end{enumerate}
\end{theorem}

\begin{remark}
\begin{enumerate}
\item[$(a)$]\quad $(1)$ of Theorem~\ref{thm0001}, which together the results of \cite{JL201}, gives a completely positive answer to the question~$(Q_1)$.
\item[$(b)$]\quad As pointed out in the very recent work \cite{JL201}, the crucial point in studying $(Q_1)$ is to obtain a good energy estimate of $m_{a,\mu}^-$ for $2<q<2+\frac{4}{N}$ such that the compactness of minimizing sequence or $(PS)$ sequence at the energy level $m_{a,\mu}^-$ still holds.  As for other concave-convex problems (cf. \cite{ABC94}) and observed in \cite{JL201}, the threshold of such compactness should be $m_{a,\mu}^++\frac{1}{N}S^{\frac{N}{2}}$.  Since $m_{a,\mu}^-$ is a mountain-pass level, the classical idea, which can be traced back to \cite{BN83}, is to use the ground state $u_{a,\mu,+}$ and the Aubin-Talanti babbles to construct a good path, whose energy can be well controlled from above to make sure that it is smaller than the threshold $m_{a,\mu}^++\frac{1}{N}S^{\frac{N}{2}}$.   This strategy is already used in \cite{JL201} to prove the existence of critical points of $\mathcal{E}_\mu(u)|_{\mathcal{S}_a}$ of mountain-pass type for $N\geq4$ and $2<q<2+\frac{4}{N}$.  Unlike \cite{JL201} in which nonradial test function composing of $u_{a,\mu,+}$ and a bubble at $\infty$ is used, here we directly use the radial superposition of $ u_{a,\mu, +}$ and the Aubin-Talenti bubble. This test function seems to be more natural and it works for all dimensions.
\item[$(c)$]\quad $(2)$ and $(3)$ of Theorem~\ref{thm0001} give partial answers to $(Q_2)$ and they are proved by observing the non-increasing of $m_{a,\mu}^-$ and suitable choices of test functions.  These two conclusions imply that the $L^2$-critical and supercritical perturbations have quite different influence on \eqref{eq0001}.  Moreover, it seems that the critical mass of ground states also exists for \eqref{eq0001} in the $L^2$-critical case.
\end{enumerate}
\end{remark}

Our next result will be devoted to the precisely asymptotic behaviors of the solutions found in \cite[Theorem~1.1]{S20}, \cite[Theorem~1.6]{JL201} and Theorem~\ref{thm0001} as $\mu\to0$.
To state this result, let us first introduce some necessary notations.  By \cite[Theorem~B]{W83}, the Gagliardo-Nirenberg inequality,
\begin{eqnarray}\label{eq0059}
\|u\|_q\leq C_{N,q}\|u\|_2^{1-\gamma_q}\|\nabla u\|_2^{\gamma_q}\quad\text{for all }u\in H^1(\bbr^N),
\end{eqnarray}
has a minimizer $\phi_0$, which satisfies
\begin{eqnarray}\label{eqnew1010}
\left\{\aligned&-\Delta \phi_0+\nu_0\phi_0=\sigma_0\phi_0^{q-1}\quad\text{in }\bbr^N,\\
&\phi_0(0)=\max_{x\in\bbr^N}\phi_0(x),\\
&\phi_0(x)>0\quad\text{in }\bbr^N,\\
&\phi_0(x)\to0\quad\text{as }|x|\to+\infty,\endaligned\right.
\end{eqnarray}
where $\nu_0=\frac{4}{N(q-2)}(1-\frac{(q-2)(N-2)}{4})$, $\sigma_0=\frac{4}{N(q-2)}$ and $C_{N,q}$ is the best constant in the Gagliardo-Nirenberg inequality.  On the other hand, the Aubin-Talanti babbles,
\begin{eqnarray}\label{eq0095}
U_\ve(x)=[N(N-2)]^{\frac{N-2}{4}}\bigg(\frac{\ve}{\ve^2+|x|^2}\bigg)^{\frac{N-2}{2}},
\end{eqnarray}
is the only solutions to the following equation:
\begin{eqnarray*}
\left\{\aligned&-\Delta u=u^{2^*-1}\quad\text{in }\bbr^N,\\
&u(0)=\max_{x\in\bbr^N}u(x),\\
&u(x)>0\quad\text{in }\bbr^N,\\
&u(x)\to0\quad\text{as }|x|\to+\infty.
\endaligned\right.
\end{eqnarray*}
Now, our second result can be stated as follows.
\begin{theorem}\label{thm0003}
Let $N\geq3$, $2<q<2^*$ and $a,\mu>0$ such that $\mu>0$ is sufficiently small.  Let $\widetilde{u}_\mu$ be the minimizer of $\mathcal{E}_\mu(u)|_{\mathcal{S}_a}$ in $\mathcal{P}_+^{a,\mu}$ and $\widehat{u}_\mu$ be the minimizer of $\mathcal{E}_\mu(u)|_{\mathcal{S}_a}$ in $\mathcal{P}_-^{a,\mu}$.  Then
\begin{enumerate}
\item[$(1)$]\quad For $2<q<2+\frac{4}{N}$, $\widetilde{w}_{a,\mu}(x)=s_\mu^{\frac{N}{2}}\widetilde{u}_\mu(s_\mu x)\to \nu_a^{\frac{1}{q-2}}\phi_0(\sqrt{\nu_a}x)$ strongly in $H^1(\bbr^N)$ as $\mu\to0$, where $\phi_0$ is the unique solution of \eqref{eqnew1010},
    \begin{eqnarray}\label{eq1008}
    \nu_a=\bigg(\frac{a^2}{\|\phi_0\|_2^2}\bigg)^{\frac{2(q-2)}{4-N(q-2)}}.
    \end{eqnarray}
and $s_\mu\sim\mu^{\frac{1}{2-q\gamma_q}}$ is the unique solution of the following system:
\begin{eqnarray}\label{eq1009}
\left\{\aligned&s_\mu^2\|\nabla \psi_{\nu_a,1}\|_2^2-\mu\gamma_q\|\psi_{\nu_a,1}\|_q^qs_\mu^{q\gamma_q}-\|\psi_{\nu_a,1}\|_{2^*}^{2^*}s_\mu^{2^*}=0,\\
&2s_\mu^2\|\nabla \psi_{\nu_a,1}\|_2^2-\mu q\gamma_q^2\|\psi_{\nu_a,1}\|_q^qs_\mu^{q\gamma_q}-2^*\|\psi_{\nu_a,1}\|_{2^*}^{2^*}s_\mu^{2^*}>0,
\endaligned\right.
\end{eqnarray}
where $\psi_{\nu_a,1}(x)=\nu_a^{\frac{1}{q-2}}\phi_0(\sqrt{\nu_a}x)$.
Moreover, up to translations and rotations, $\widetilde{u}_\mu$ is the unique ground state of \eqref{eq0001} for $\mu>0$ sufficiently small.
\item[$(2)$]\quad For $N\geq5$, $\widehat{u}_\mu\to U_{\ve_0}$ strongly in $H^1(\bbr^N)$ as $\mu\to0$, where $U_{\ve_0}$ is the Aubin-Talanti babble satisfying $\|U_{\ve_0}\|_2^2=a^2$.  Moreover, up to translations and rotations, $\widehat{u}_\mu$ is the unique minimizer of $\mathcal{E}_\mu(u)|_{\mathcal{S}_a}$ in $\mathcal{P}_-^{a,\mu}$ for $\mu>0$ sufficiently small.
\item[$(3)$]\quad For $N=3,4$,
$\widehat{w}_{a,\mu}(x)=\ve_{\mu}^{\frac{N-2}{2}}\widehat{u}_\mu(\ve_{\mu} x)\to U_{\ve_0}$ strongly in $D^{1,2}(\bbr^N)$ for some $\ve_0>0$ as $\mu\to0$ up to a subsequence,
where $\ve_{\mu}$ satisfies
\begin{eqnarray*}
\mu\sim\left\{\aligned&\ve_{\mu}^{6-q}e^{-2\ve_{\mu}^{-2}},\quad N=4, 2<q<4,\\
&\ve_{\mu}^{\frac{q}{2}-1},\quad N=3, 3<q<6,\\
&\frac{\ve_{\mu}^{\frac{1}{2}}}{\ln(\frac{1}{\ve_\mu})},\quad N=3, q=3,\\
&\ve_\mu^{5-\frac{3q}{2}},\quad N=3, 2<q<3.\endaligned\right.
\end{eqnarray*}
\end{enumerate}
\end{theorem}

\begin{remark}
\begin{enumerate}
\item[$(1)$]\quad The precise asymptotic behaviors of $\widetilde{u}_\mu$ and $\widehat{u}_\mu$ stated in $(1)$ and $(2)$ of Theorem~\ref{thm0003} are captured by comparing the energy values and norms by full using the variational formulas of $\widetilde{u}_\mu$ and $\widehat{u}_\mu$, and minimizers of the Gagliardo-Nirenberg inequality and the Aubin-Talanti bubbles.  In this argument, the unique determination of minimizers of the Gagliardo--Nirenberg inequality~\eqref{eq0059} for $2<q<2+\frac{4}{N}$ and Aubin-Talanti bubbles for $N\geq5$ in $\mathcal{S}_a$, respectively, is crucial.  Moreover, $(2)$ of Theorem~\ref{thm0003} also gives a positive answer to Soave's conjecture on $(Q_3)$.
\item[$(2)$]\quad For the local uniqueness, the standard strategy is to assume the contrary and obtain a contradiction by full using the non-degeneracy of minimizers of the Gagliardo--Nirenberg inequality and Aubin-Talanti bubbles in passing to the limit (cf. \cite{GLWZ19,DLY16}), which is powerful in studying problems with potentials.  Since \eqref{eq0001} is autonomous, we can use a different method, based on the precisely asymptotic behaviors of $\widetilde{u}_\mu$ and the implicit function theorem, to prove the local uniqueness of $\widetilde{u}_\mu$ in a more direct way.  It is worth pointing out that our method is also based on the non-degenerate of minimizers of the Gagliardo--Nirenberg inequality.  For $\widehat{u}_\mu$, we remark that since the linear operator of the limit equation is different from that of \eqref{eq0001}, our direct methods, based on implicit function theorem, is invalid.  Thus, we will still use the standard method, that is to assume the contrary and obtain a contradiction by full using the non-degeneracy of Aubin-Talanti bubbles.
\item[$(3)$]\quad Since we loss the $L^2$-integrability of the Aubin-Talanti babbles $\{U_\ve\}$ for $N=3,4$, the asymptotic behavior of $\widehat{u}_\mu$ as $\mu\to0$ for $N=3,4$ is much weaker than that of $N\geq5$ in the sense that, the convergence is only for subsequences, which also leads us to loss the local uniqueness of $\widehat{u}_\mu$ for $\mu>0$ sufficiently small in these two cases.  We also remark that since we loss the $L^2$-integrability of the Aubin-Talanti babbles $\{U_\ve\}$ for $N=3,4$, the asymptotic behavior of $\widehat{u}_\mu$ can not be obtained by merely using variational arguments to compare the energy values and norms as that for $(2)$ of Theorem~\ref{thm0003}.  Thus, to capture the precisely asymptotic behavior of $\widehat{u}_\mu$, we drive some uniformly pointwise estimates of $\widehat{u}_\mu$ by the maximum principle (cf. \cite{DPG13}) and some ODE technique used in \cite{AP86} (see also \cite{GS03,KP89}).  With these additional estimates, we obtain the precisely asymptotic behavior of $\widehat{u}_\mu$ for $N=3,4$.  It is worth pointing out that, in the case $N=3$ and $2<q<3$, since the nonlinearity decays too slow at infinity, we need to further employ the bootstrapping argument to drive the desired estimates.
\end{enumerate}
\end{remark}

Our final result is devoted to the asymptotic behavior of the minimizers of $\mathcal{E}_\mu(u)|_{\mathcal{S}_a}$ in $\mathcal{P}_-^{a,\mu}$ as $\mu$ close to its upper-bound in the cases of $2+\frac{4}{N}\leq q<2^*$.  It can be stated as follows.
\begin{theorem}\label{thm0002}
Assume $N\geq3$, $2+\frac{4}{N}\leq q<2^*$ and $\mu,a>0$.  Let $\widehat{u}_\mu$ be the minimizer of $\mathcal{E}_\mu(u)|_{\mathcal{S}_a}$ in $\mathcal{P}_-^{a,\mu}$, found in \cite[Theorem~1.1]{S20} for $q=2+\frac{4}{N}$ with $0<\mu a^{q-q\gamma_q}<\alpha_{N,q}$ and found in Theorem~\ref{thm0001} for $2+\frac{4}{N}<q<2^*$ with all $\mu>0$.  Then
\begin{enumerate}
\item[$(1)$]\quad  For $q=2+\frac{4}{N}$, $\widehat{v}_\mu=(\frac{a}{\|\phi_0\|_2})^{\frac{N-2}{2}}s_\mu^{\frac{N}{2}}\widehat{u}_\mu(\frac{a}{\|\phi_0\|_2}s_\mu x)\to (\nu_a')^{\frac{1}{q-2}}\phi_0(\sqrt{\nu_a'}x)$ strongly in $H^1(\bbr^N)$ as $\mu\to \alpha_{N,q,a}$ up to a subsequence, where $\alpha_{N,q,a}=a^{q\gamma_q-q}\alpha_{N,q}$ for some $\nu_a'>0$ and $s_\mu=(1-\frac{\mu}{\alpha_{N,q,a}})^{-\frac{N-2}{4}}$.
\item[$(2)$]\quad For $2+\frac{4}{N}<q<2^*$, $\widehat{v}_\mu=s_\mu^{\frac{N}{2}}\widehat{u}_\mu(s_\mu x)\to \nu_a^{\frac{1}{q-2}}\phi_0(\sqrt{\nu_a}x)$ strongly in $H^1(\bbr^N)$ as $\mu\to +\infty$, where $s_\mu=\mu^{\frac{1}{q\gamma_q-2}}$.  Moreover, up to translations and rotations, $\widehat{u}_\mu$ is also the unique ground state of \eqref{eq0001} for $\mu>0$ sufficiently large.
\end{enumerate}
\end{theorem}

\begin{remark}
\begin{enumerate}
\item The ideas in proving Theorem~\ref{thm0002} are similar to that of Theorem~\ref{thm0003}.  However, in the $L^2$-critical case $q=2+\frac{4}{N}$, the convergence of $\widehat{u}_\mu$ is much weaker than that in the $L^2$-supcritical case $2+\frac{4}{N}<q<2^*$ in the sense that, it only holds for subsequences.  The main reason is that in the $L^2$-critical case $q=2+\frac{4}{N}$, we have $\|\varphi\|_2^2=const.$ for all $\varphi$ being a minimizer of the Gagliardo--Nirenberg inequality~\eqref{eq0059}.  Thus, the precise mass $\|\widehat{u}_\mu\|_2^2=a^2$ is invalid in determining a unique minimizer of the Gagliardo--Nirenberg inequality~\eqref{eq0059} in the case $q=2+\frac{4}{N}$.  Moreover, unlike the studies for problems with homogeneous nonlinearities (cf. \cite{GLW17,GLWZ19}), combining nonlinearities ($L^2$-critical and $L^2$-supercritical) of \eqref{eq0001} makes the asymptotic behavior of $\widehat{u}_\mu$ to be more complicated, which also make us loss the local uniqueness of $\widehat{u}_\mu$ for $\mu>0$ close to its upper bound in this case.  Indeed, as $\mu$ goes to its upper bound in the $L^2$-critical case, comparing with the studies for problems with homogeneous nonlinearities, the Sobolev critical term of \eqref{eq0001} is an additionally inhomogenous perturbation in passing to the limit, which makes the oscillations occurring.
\end{enumerate}
\end{remark}

\noindent{\bf\large Notations.} Throughout this paper, $C$ and $C'$ are indiscriminately used to denote various absolutely positive constants.  $a\sim b$ means that $C'b\leq a\leq Cb$ and $a\lesssim b$ means that $a\leq Cb$.

\section{Asymptotic behavior of $u_{a,\mu,+}$}
By \cite[Theorem~1.1]{S20}, $m_{a,\mu}^+$  can always be attained by some $u_{a,\mu,+}$ for $2<q<2+\frac{4}{N}$ and $\mu a^{q-q\gamma_q}<\alpha_{N,q}$, where $m_{a,\mu}^+$ is given by \eqref{eq0045} and $u_{a,\mu,+}$ is real valued, positive, radially symmetric and radially decreasing.  Our goal in this section is to give an asymptotic behavior of $u_{a,\mu,+}$ as $\mu\to0$, which is more precisely than that in \cite[Theorem~1.4]{S20}, and capture the precisely decaying rate of $u_{a,\mu,+}$ as $\mu\to0$.  We recall that by \cite[Theorem~1.1]{S20}, $u_{a,\mu,+}$ is a solution of \eqref{eq0001} for some $\lambda_{a,\mu,+}<0$.  To simplify the notation, we shall denote $u_{\mu,+}=u_{a,\mu,+}$ and $\lambda_{\mu,+}=\lambda_{a,\mu,+}$, since we will fix $a>0$ in what follows.  Let us begin with
\begin{lemma}\label{lemma0008}
Let $2<q<2+\frac{4}{N}$.  Then $-\lambda_{\mu,+}\sim\|\nabla u_{\mu,+}\|_2^2\sim\mu^{\frac{2}{2-q\gamma_q}}$ as $\mu\to0$.
\end{lemma}
\begin{proof}
Since $u_{\mu,+}\in\mathcal{P}_+^{a,\mu}$, we have
\begin{eqnarray}\label{eq0020}
\|\nabla u_{\mu,+}\|_2^2=\mu\gamma_q\|u_{\mu,+}\|_q^q+\|u_{\mu,+}\|_{2^*}^{2^*}
\end{eqnarray}
and
\begin{eqnarray*}
2\|\nabla u_{\mu,+}\|_2^2>\mu q\gamma_q^2\|u_{\mu,+}\|_q^q+2^*\|u_{\mu,+}\|_{2^*}^{2^*}.
\end{eqnarray*}
It follows from the Gagliardo--Nirenberg inequality that
\begin{eqnarray*}
\|\nabla u_{\mu,+}\|_2^2\lesssim\mu\|u_{\mu,+}\|_q^q \lesssim\mu\|\nabla u_{\mu,+}\|_2^{q\gamma_q},
\end{eqnarray*}
which together with $q\gamma_q<2$ for $2<q<2+\frac{4}{N}$, implies
\begin{eqnarray}\label{eq0021}
\|\nabla u_{\mu,+}\|_2^2\lesssim\mu^{\frac{2}{2-q\gamma_q}}.
\end{eqnarray}
Thus, by \eqref{eq0020} and \eqref{eq0021}, we also have
\begin{eqnarray}\label{eq0050}
\mu\|u_{\mu,+}\|_q^q\lesssim\mu^{\frac{2}{2-q\gamma_q}}.
\end{eqnarray}
Let us define
\begin{eqnarray}\label{eq0069}
V_{\ve}(x)=U_\ve(x)\varphi(R_\ve^{-1}x)
\end{eqnarray}
where $U_\varepsilon(x)$ is the Aubin-Talanti babbles given by \eqref{eq0095} and
$\varphi\in C_0^\infty(\mathbb{R}^N)$ is a radial cut-off function with $\varphi\equiv 1$ in $B_1$, $\varphi\equiv 0$ in $B_2^c$, and $R_\ve$ is chosen such that $V_{\ve}\in\mathcal{S}_a$.  More precisely, for $N\geq5$, we choose $\ve=\ve_0$ and $R_{\ve_0}=+\infty$ such that $V_{\ve_0}=U_{\ve_0}\in\mathcal{S}_a$ while for $N=3,4$, we choose $\ve>0$ sufficiently small and then in the later two cases, we have
\begin{eqnarray}\label{eq0023}
a^2=\int_{\bbr^N}(U_\ve(x)\varphi(R_\ve^{-1}x))^2\sim\ve^{2}\int_1^{R_\ve\ve^{-1}}r^{3-N}\sim\left\{\aligned\ve^{2}\ln(R_\ve\ve^{-1}), \text{ for }N=4,\\
\ve R_\ve, \text{ for }N=3,
\endaligned\right.
\end{eqnarray}
which implies $R_\ve\ve^{-1}\to+\infty$ as $\ve\to0$.
Then it is well-known (cf. \cite[(4.2)--(4.5)]{MM14} or \cite[Chapter III]{S00}) that
\begin{eqnarray}
\|\nabla V_{\ve}\|_2^2=S^{\frac{N}{2}}+O((R_\ve\ve^{-1})^{2-N}),\quad \|V_{\ve}\|_{2^*}^{2^*}=S^{\frac{N}{2}}+O((R_\ve\ve^{-1})^{-N})\label{eq0036}
\end{eqnarray}
for $\ve>0$ sufficiently small, which implies
\begin{eqnarray}\label{eq0022}
\|\nabla V_{\ve}\|_2^2\sim S^{\frac{N}{2}}\sim\|V_{\ve}\|_{2^*}^{2^*}
\end{eqnarray}
for $\ve>0$ sufficiently small.  Now, we fix $\ve=\ve_0$ and $R_{\ve_0}=+\infty$ for $N=5$, and fix $\ve>0$ sufficiently small and choose $R_\ve$ as that in \eqref{eq0023} for $N=3,4$ such that \eqref{eq0022} holds for all $N\geq3$.  By \cite[Lemma~4.2]{S20}, there exists $t(\mu)>0$ such that $(V_{\ve})_{t(\mu)}\in\mathcal{P}_+^{a,\mu}$  for $\mu>0$ sufficiently small, where
\begin{eqnarray*}
(V_{\ve})_{t(\mu)}=[t(\mu)]^{\frac{N}{2}}V_{\ve}(t(\mu)x).
\end{eqnarray*}
Then
\begin{eqnarray*}
[t(\mu)]^2\|\nabla V_\ve\|_2^2=\mu\gamma_q\|V_\ve\|_q^q[t(\mu)]^{q\gamma_q}+\|V_\ve\|_{2^*}^{2^*}[t(\mu)]^{2^*}
\end{eqnarray*}
and
\begin{eqnarray*}
2[t(\mu)]^2\|\nabla V_\ve\|_2^2>\mu q\gamma_q^2\|V_\ve\|_q^q[t(\mu)]^{q\gamma_q}+2^*\|V_\ve\|_{2^*}^{2^*}[t(\mu)]^{2^*}.
\end{eqnarray*}
Since $q\gamma_q<2$ for $2<q<2+\frac{4}{N}$, by
\begin{eqnarray*}
(2^*-2)[t(\mu)]^2\|\nabla V_\ve\|_2^2<\mu(2^*-q\gamma_q)\gamma_q\|V_\ve\|_q^q[t(\mu)]^{q\gamma_q},
\end{eqnarray*}
it is easy to see that $t(\mu)\to0$ as $\mu\to0$ for all $N\geq3$.  It follows that
\begin{eqnarray*}
[t(\mu)]^2\sim\mu [t(\mu)]^{q\gamma_q}\quad\text{as }\mu\to0,
\end{eqnarray*}
which implies $t(\mu)\sim\mu^{\frac{1}{2-q\gamma_q}}$ as $\mu\to0$.  Thus, by $q\gamma_q<2$ for $2<q<2+\frac{4}{N}$ once more,
\begin{eqnarray*}
\mathcal{E}_\mu((V_{\ve})_{t(\mu)})=(\frac{1}{2}-\frac{1}{q\gamma_q})\|\nabla V_{\ve}\|_2^2[t(\mu)]^2+(\frac{1}{q\gamma_q}-\frac{1}{2^*})\|V_{\ve}\|_{2^*}^{2^*}[t(\mu)]^{2^*}\sim-\mu^{\frac{2}{2-q\gamma_q}}.
\end{eqnarray*}
Therefore, by $\mathcal{E}_\mu((V_{\ve})_{t(\mu)})\geq m_{a,\mu}^+$ and $m_{a,\mu}^+\gtrsim-\mu\|u_{\mu,+}\|_q^q$, we have
\begin{eqnarray*}
\mu\|u_{\mu,+}\|_q^q\gtrsim\mu^{\frac{2}{2-q\gamma_q}},
\end{eqnarray*}
which together with \eqref{eq0050}, implies
\begin{eqnarray*}
\mu\|u_{\mu,+}\|_q^q\sim\mu^{\frac{2}{2-q\gamma_q}}.
\end{eqnarray*}
By the regularity of $u_{\mu,+}$ and the Pohozaev identity, $\lambda_{\mu,+}\sim-\mu\|u_{\mu,+}\|_q^q$, and by \eqref{eq0021} and $u_{\mu,+}\in\mathcal{P}_+^{a,\mu}$, $\|\nabla u_{\mu,+}\|_2^2\sim\mu\|u_{\mu,+}\|_q^q$.  Therefore,
\begin{eqnarray*}
-\lambda_{\mu,+}\sim\|\nabla u_{\mu,+}\|_2^2\sim\mu^{\frac{2}{2-q\gamma_q}}
\end{eqnarray*}
as $\mu\to0$.  It completes the proof.
\end{proof}

By the well-known uniqueness result (cf. \cite{K89}) and the scaling invariance of \eqref{eqnew1010},
\begin{eqnarray*}
\phi_0(x)=\bigg(\frac{\nu_0}{\sigma_0}\bigg)^{\frac{1}{q-2}}w(\sqrt{\nu_0}x),
\end{eqnarray*}
where $w$ is the unique solution of the following equation:
\begin{eqnarray}\label{eq0024}
\left\{\aligned&-\Delta u+u=u^{q-1}\quad\text{in }\bbr^N,\\
&u(0)=\max_{x\in\bbr^N}u(x),\\
&u(x)>0\quad\text{in }\bbr^N,\\
&u(x)\to0\quad\text{as }|x|\to+\infty,\endaligned\right.
\end{eqnarray}
A direct calculation also shows that
\begin{eqnarray}\label{eq1005}
\psi_{\nu,\sigma}(x)=(\frac{\nu}{\sigma})^{\frac{1}{q-2}}\phi_0(\sqrt{\frac{\nu}{\sigma}}x)
\end{eqnarray}
for $\nu,\sigma>0$ are all minimizers of the Gagliardo-Nirenberg inequality~\eqref{eq0059}.  Let $\nu_a$ be given by \eqref{eq1008}, then for $q\not=2+\frac{4}{N}$,
$\psi_{\nu_a,1}\in\mathcal{S}_a$ and $\psi_{\nu_a,1}$ is a minimizer of the Gagliardo--Nirenberg inequality, that is,
\begin{eqnarray}\label{eq0010}
\|\psi_{\nu_a,1}\|_q^q=C_{N,q}^qa^{q-q\gamma_q}\|\nabla\psi_{\nu_a,1}\|_2^{q\gamma_q}.
\end{eqnarray}
For the sake of simplicity, we re-denote $\psi_{a}=\psi_{\nu_a,1}$.
\begin{proposition}\label{prop0004}
Let $2<q<2+\frac{4}{N}$.  Then $w_{\mu,+}\to \psi_a$ strongly in $H^1(\bbr^N)$ as $\mu\to0$, where $w_{\mu,+}=s_\mu^{-\frac{N}{2}}u_{\mu,+}(s_\mu^{-1} x)$ with $s_\mu$ being the unique solution of \eqref{eq1009}.
Moreover, up to translations and rotations, $u_{\mu,+}$ is the unique ground state of \eqref{eq0001} for $\mu>0$ sufficiently small.
\end{proposition}
\begin{proof}
Since $\psi_{a}\in\mathcal{S}_a$, by \cite[Lemma~4.2]{S20}, there exists a unique $s_\mu>0$ such that $(\psi_{a})_{s_\mu}\in\mathcal{P}_+^{a,\mu}$ for $\mu>0$ sufficiently small where $(\psi_{a})_{s_\mu}=s_\mu^{\frac{N}{2}}\psi_a(s_\mu x)$.  That is,
\begin{eqnarray}\label{eq1002}
s_\mu^2\|\nabla \psi_{a}\|_2^2=\mu\gamma_q\|\psi_{a}\|_q^qs_\mu^{q\gamma_q}+\|\psi_{a}\|_{2^*}^{2^*}s_\mu^{2^*}
\end{eqnarray}
and
\begin{eqnarray}\label{eq1003}
2s_\mu^2\|\nabla \psi_{a}\|_2^2>\mu q\gamma_q^2\|\psi_{a}\|_q^qs_\mu^{q\gamma_q}+2^*\|\psi_{a}\|_{2^*}^{2^*}s_\mu^{2^*}.
\end{eqnarray}
As that in the proof of Lemma~\ref{lemma0008}, we have
\begin{eqnarray}\label{eq1000}
\|\nabla (\psi_{a})_{s_\mu}\|_2^{2-q\gamma_q}<C_{N,q}^q\gamma_q\mu a^{q-q\gamma_q}\frac{2^*-q\gamma_q}{2^*-2}.
\end{eqnarray}
Since $u_{\mu,+}\in\mathcal{P}_+^{a,\mu}$, we also have
\begin{eqnarray}\label{eq1001}
\|\nabla u_{\mu,+}\|_2^{2-q\gamma_q}<C_{N,q}^q\gamma_q\mu a^{q-q\gamma_q}\frac{2^*-q\gamma_q}{2^*-2}.
\end{eqnarray}
Now, using $(\psi_{a})_{s_\mu}$ as a test function of $m_{a,\mu}^+$ and by \eqref{eq0010},
\begin{eqnarray*}
m_{a,\mu}^+\leq\mathcal{E}_\mu((\psi_{a})_{s_\mu})=\frac{1}{N}\|\nabla (\psi_{a})_{s_\mu}\|_2^{2}-\frac{\mu a^{q-q\gamma_q}C_{N,q}^q}{q}(1-\frac{q\gamma_q}{2^*})\|\nabla (\psi_{a})_{s_\mu}\|_2^{q\gamma_q}.
\end{eqnarray*}
By the Gagliardo--Nirenberg inequality~\eqref{eq0059},
\begin{eqnarray*}
m_{a,\mu}^+=\mathcal{E}_\mu(u_{\mu,+})\geq\frac{1}{N}\|\nabla u_{\mu,+}\|_2^{2}-\frac{\mu a^{q-q\gamma_q}C_{N,q}^q}{q}(1-\frac{q\gamma_q}{2^*})\|\nabla u_{\mu,+}\|_2^{q\gamma_q}.
\end{eqnarray*}
Let us consider the function
\begin{eqnarray*}
f(t)=\frac{1}{N}t^{2}-\frac{\mu a^{q-q\gamma_q}C_{N,q}^q}{q}(1-\frac{q\gamma_q}{2^*})t^{q\gamma_q}.
\end{eqnarray*}
A direct calculation shows that $f(t)$ is strictly decreasing in $(0, t_0)$, where
\begin{eqnarray*}
t_0=\bigg(C_{N,q}^q\gamma_q\mu a^{q-q\gamma_q}\frac{2^*-q\gamma_q}{2^*-2}\bigg)^{\frac{1}{2-q\gamma_q}}.
\end{eqnarray*}
Thus, by \eqref{eq1000} and \eqref{eq1001},
\begin{eqnarray}\label{eq1007}
\|\nabla u_{\mu,+}\|_2^{2}\geq\|\nabla (\psi_{a})_{s_\mu}\|_2^{2}.
\end{eqnarray}
By \eqref{eq1002} and \eqref{eq1003}, we can use similar arguments as that used in the proof of Lemma~\ref{lemma0008} to show that $s_\mu\sim\mu^{\frac{1}{2-q\gamma_q}}$ as $\mu\to0$.  It then follows from \eqref{eq1005} and \eqref{eq1002} that
\begin{eqnarray*}
s_\mu=(1+o_\mu(1))\bigg(\frac{\mu\gamma_q\|\psi_{a}\|_q^q}{\|\nabla \psi_{a}\|_2^2}\bigg)^{\frac{1}{2-q\gamma_q}}=(1+o_\mu(1))\bigg(\frac{\mu\gamma_q\|\phi_0\|_q^q}{\|\nabla \phi_0\|_2^2}\bigg)^{\frac{1}{2-q\gamma_q}}.
\end{eqnarray*}
Since by the Pohozaev identity satisfied by $\phi$, we have $\frac{1}{N}\|\nabla \phi_0\|_2^2=\frac{(q-2)\sigma_0}{2q}\|\phi_0\|_q^q$.  By \eqref{eq0008}, $s_\mu=[(\frac{1}{\sigma_0}+o_\mu(1))\mu]^{\frac{1}{2-q\gamma_q}}$.
Let
\begin{eqnarray*}
w_{\mu,+}=s_\mu^{-\frac{N}{2}}u_{\mu,+}(s_\mu^{-1} x).
\end{eqnarray*}
Since $u_{\mu,+}$ satisfies \eqref{eq0001},
$w_{\mu,+}$ satisfies the following equation:
\begin{eqnarray}\label{eq0052}
-\Delta w_{\mu,+}=\lambda_{\mu,+} s_\mu^{-2} w_{\mu,+}+\mu s_\mu^{-2+\frac{N}{2}(q-2)}w_{\mu,+}^{q-1}+s_\mu^{-2+\frac{N}{2}(2^*-2)}w_{\mu,+}^{2^*-1}.
\end{eqnarray}
By Lemma~\ref{lemma0008} and
\begin{eqnarray*}
\int_{\bbr^N}w_{\mu,+}^2=\int_{\bbr^N}u_{\mu,+}^2\equiv a^2,
\end{eqnarray*}
we have
\begin{eqnarray*}
\|\nabla w_{\mu,+}\|_2^2+\|w_{\mu,+}\|_2^2= s_\mu^{-2}\|\nabla u_{\mu,+}\|_2^2+a^2\sim1.
\end{eqnarray*}
Therefore, $\{w_{\mu,+}\}$ is bounded in $H^1(\bbr^N)$.  It follows that $w_{\mu,+}\rightharpoonup w_*$ weakly in $H^1(\bbr^N)$ as $\mu\to0$ up to a subsequence.  Note that $w_{\mu,+}$ is radial, by Struss's radial lemma (cf. \cite[Lemma A.IV, Theorem A.I']{BL83} or \cite[Lemma~3.1]{MM14}) and the Sobolev embedding theorem, $w_{\mu,+}\to w_*$ strongly in $L^q(\bbr^N)$ as $\mu\to0$ up to a subsequence.  By Lemma~\ref{lemma0008} once more, $\{\lambda_{\mu,+}\mu^{\frac{-2}{2-q\gamma_q}}\}$ is bounded.  Thus, $\lambda_{\mu,+}\mu^{\frac{-2}{2-q\gamma_q}}\to \alpha_*$ as $\mu\to0$ up to a subsequence.  On the other hand, by $q\gamma_q<2$ for $2<q<2+\frac{4}{N}$,
\begin{eqnarray*}
s_\mu^{-2+\frac{N}{2}(2^*-2)}\sim\mu^{\frac{2^*-2}{2-q\gamma_q}}\to0
\end{eqnarray*}
as $\mu\to0$.  Now, using \eqref{eq1007} and \eqref{eq0052}, it is standard to show that $w_{\mu,+}\to w_*$ strongly in $H^1(\bbr^N)$ as $\mu\to0$ up to a subsequence, where $w_*$ is the unique solution of the following equation:
\begin{eqnarray}\label{eq0053}
\left\{\aligned&-\Delta u+\alpha_*u=\sigma_0u^{q-1}\quad\text{in }\bbr^N,\\
&u(0)=\max_{x\in\bbr^N}u(x),\\
&u(x)>0\quad\text{in }\bbr^N,\\
&u(x)\to0\quad\text{as }|x|\to+\infty,\endaligned\right.
\end{eqnarray}
by the well-known uniqueness result (cf. \cite{K89}) and the scaling invariance of \eqref{eqnew1010}, $w_*(x)=(\frac{\alpha_*}{\sigma_0})^{\frac{1}{q-2}}w(\sqrt{\alpha_*}x)$, where $w$ is the unique solution of \eqref{eq0024}.  It follows from $\|w_{\mu,+}\|_2^2=a^2$ and the strong convergence of $\{w_{\mu,+}\}$ in $H^1(\bbr^N)$ that $\|w_*\|_2^2=a^2$, which implies $\alpha_*=\nu_{a}\nu_0$ where $\nu_a$ is given by \eqref{eq1008}.  Thus, $w_*=\psi_a$.  Since $\psi_a$ is unique, $w_{\mu,+}\to \psi_a$ strongly in $H^1(\bbr^N)$ as $\mu\to0$.  The system~\eqref{eq1009} directly comes from \eqref{eq1002} and \eqref{eq1003}.
It remains to prove the local uniqueness of $u_{\mu,+}$ for $\mu>0$ sufficiently small.  Let us consider the following system:
\begin{eqnarray}\label{eq0025}
\left\{\aligned&\mathcal{F}(w,\alpha,\beta,\gamma)=\Delta w-\alpha\nu_0 w+\beta w^{q-1}+\gamma w^{2^*-1},\\
&\mathcal{G}(w,\alpha,\beta,\gamma)=\|w\|_2^2-a^2,
\endaligned\right.
\end{eqnarray}
where $\alpha,\beta,\gamma>0$ are parameters.
It is easy to see that $\mathcal{F}(\psi_a,\nu_a,\sigma_0,0)=0$ and $\mathcal{G}(\psi_a,\nu_a,\sigma_0,0)=0$.  Let
\begin{eqnarray*}
\mathcal{L}(\psi_a,\nu_a,\sigma_0,0)=\left(\aligned\partial_w\mathcal{F}(\psi_a,\nu_a,\sigma_0,0)\quad\partial_\alpha\mathcal{F}(\psi_a,\nu_a,\sigma_0,0)\\
\partial_w\mathcal{G}(\psi_a,\nu_a,\sigma_0,0)\quad \partial_\alpha\mathcal{G}(\psi_a,\nu_a,\sigma_0,0)\endaligned\right)
\end{eqnarray*}
be the linearization of the system~\eqref{eq0025} at $(\psi_a,\nu_a,\sigma_0,0)$ in $H^1(\bbr^N)\times\bbr$, that is,
\begin{eqnarray*}
\partial_w\mathcal{F}(\psi_a,\nu_a,\sigma_0,0)=\Delta-\nu_a\nu_0+(q-1)\sigma_0\psi_a^{q-2},\quad\partial_\alpha\mathcal{F}(\psi_a,\nu_a,\sigma_0,0)=-\nu_0\psi_a
\end{eqnarray*}
and
\begin{eqnarray*}
\partial_w\mathcal{G}(\psi_a,\nu_a,\sigma_0,0)=2\psi_a,\quad\partial_\alpha\mathcal{G}(\psi_a,\nu_a,\sigma_0,0)=0.
\end{eqnarray*}
Then $\mathcal{L}(\psi_a,\nu_a,\sigma_0,0)[(\phi,\tau)]=0$ if and only if
\begin{eqnarray*}\label{eq0026}
\left\{\aligned&\Delta \phi-\nu_a\nu_0\phi+(q-1)\sigma_0\psi_a^{q-2}\phi-\tau \nu_0\psi_a=0,\\
&\int_{\bbr^N}\psi_a\phi=0.
\endaligned\right.
\end{eqnarray*}
Let us consider the following system:
\begin{eqnarray}\label{eq1126}
\left\{\aligned&\Delta \phi-\nu_a\nu_0\phi+(q-1)\sigma_0\psi_a^{q-2}\phi-\tau \nu_0\psi_a=g,\\
&\int_{\bbr^N}\psi_a\phi=b,
\endaligned\right.
\end{eqnarray}
where $(g,b)\in H_{rad}^1(\bbr^N)\times\bbr$ with
\begin{eqnarray*}
H_{rad}^1(\bbr^N)=\{u\in H^1(\bbr^N)\mid \text{$u$ is radial}\}.
\end{eqnarray*}
Then $\phi=\phi_g+\tau\nu_0\phi_a$, where $\phi_g$ and $\phi_a$ satisfies
\begin{eqnarray}\label{eq0027}
\Delta \phi_g-\nu_a\nu_0\phi_g+(q-1)\sigma_0\psi_a^{q-2}\phi_g=g
\end{eqnarray}
and
\begin{eqnarray}\label{eq1127}
\Delta \phi_a-\nu_a\nu_0\phi_a+(q-1)\sigma_0\psi_a^{q-2}\phi_a=\psi_a,
\end{eqnarray}
respectively.  By \cite[(5.2) and (5.3)]{W99}, $\phi_a=\frac{1}{q-2}\psi_a+\frac12 (x\cdot\nabla\psi_a)$ and
\begin{eqnarray*}
\int_{\bbr^N}\phi_a\psi_a=(\frac{1}{q-2}-\frac{4}{N})\|\psi_a\|_2^2\not=0
\end{eqnarray*}
since $q\not=2+\frac{4}{N}$.
Thus, the unique solution of \eqref{eq1126} is given by $(\phi_g+\tau_{b,g}\nu_0\phi_a,\tau_{b,g})$
where
\begin{eqnarray*}
\tau_{b,g}=\frac{b-\int_{\bbr^N}\phi_g\psi_a}{\nu_0\int_{\bbr^N}\phi_a\psi_a}.
\end{eqnarray*}
Since $q<2^*$, it is well-known that $\psi_a$ is nondegenerate (cf. \cite[Theorem~2.12]{LN88} and \cite[Lemma~4.2]{NT93}).  Thus, by $q<2^*$, \eqref{eq0027} only has zero solution in $H_{rad}^1(\bbr^N)$ for $g=0$, which implies the linear operator $\mathcal{L}(\psi_a,\nu_a,\sigma_0,0):H_{rad}^1(\bbr^N)\times\bbr\to H_{rad}^1(\bbr^N)\times\bbr$ is bijective.  Moreover, it is standard to show that
\begin{eqnarray*}
|\tau_{b,g}|+\|\phi_g+\tau_{b,g}\phi_a\|_{H^1}\lesssim |b|+\|g\|_{H^1}.
\end{eqnarray*}
Now,
by the implicit function theorem, there exists a unique $C^1$-curve $(w^{\beta,\gamma}, \alpha^{\beta,\gamma})$ in $H_{rad}^1(\bbr^N)\times\bbr^3$ for $|\beta-\sigma_0|<<1$ and $|\gamma|<<1$ such that $(w^{\sigma_0,0}, \alpha^{\sigma_0,0})=(\psi_a,\nu_a)$, and
\begin{eqnarray*}
\mathcal{F}(w^{\beta,\gamma},\alpha^{\beta,\gamma},\beta,\gamma)\equiv0,\quad\mathcal{G}(w^{\beta,\gamma},\alpha^{\beta,\gamma},\beta,\gamma)\equiv0.
\end{eqnarray*}
We recall that $w_{\mu,+}$ is radial and satisfies \eqref{eq0052}, and $w_{\mu,+}\to \psi_a$ strongly in $H^1(\bbr^N)$ as $\mu\to0$ with $\|w_{\sigma,+}\|_2^2=a^2$, thus, by the uniqueness of $s_\mu$ determined by \eqref{eq1009}, we must have $w_{\mu,+}=w^{\beta(\mu),\gamma(\mu)}$ for $\beta(\mu)=\mu s_\mu^{-2+\frac{N}{2}(q-2)}$ and $\gamma(\mu)=s_\mu^{-2+\frac{N}{2}(2^*-2)}$ with $\mu>0$ sufficiently small.  On the other hand, if $\widetilde{u}_{\mu,+}$ is another ground state of \eqref{eq0001} with some $\widetilde{\lambda}_{\mu,+}\in\bbr$ for $\mu>0$ sufficiently small, then by \cite[Theorem~1.3]{S20}, $\widetilde{u}_{\mu,+}=e^{i\theta}\widehat{u}_{\mu,+}$ where $\theta$ is a constant and $\widehat{u}_{\mu,+}$ is real valued and positive.
Since by the Pohozaev identity, we always have $\widetilde{\lambda}_{\mu,+}<0$.  By
applying the well-known Gidas-Ni-Nirenberg theorem (cf. \cite{GNN81}), $\widetilde{u}_{\mu,+}$ must be radially symmetric.  Now, by running the arguments as used above once more, we know that $\widehat{w}_{\mu,+}=s_\mu^{-\frac{N}{2}}\widehat{u}_{\mu,+}(s_\mu^{-1} x)\to \psi_a$ strongly in $H^1(\bbr^N)$ as $\mu\to0^+$ with $\|w_{\sigma,+}\|_2^2=a^2$.
It follows from the uniqueness of $w^{\beta(\mu),\gamma(\mu)}$ that $\widehat{u}_{\mu,+}=u_{\mu,+}$ for $\mu>0$ sufficiently small.  Thus, $u_{\mu,+}$ is the unique ground state of \eqref{eq0001} for $\mu>0$ sufficiently small up to translations and rotations.
\end{proof}

\section{Existence and nonexistence of $u_{a,\mu,-}$}
In this section, we shall mainly study the question~$(Q_1)$.  Since in the very recent work \cite{JL201}, the question~$(Q_1)$ has been solved for $N\geq4$.  we only consider the case $N=3$ and prove that $m_{a,\mu}^-$ can also be attained by some $u_{a,\mu,-}$ for $2<q<2+\frac{4}{N}$ in the case $N=3$ under some additional assumptions, where $m_{a,\mu}^-$ is also given by \eqref{eq0045} and $u_{a,\mu,-}$ is also real valued, positive, radially symmetric and radially decreasing.  The crucial point in this study is the following energy estimates.
\begin{lemma}\label{lemma0002}
Let $N=3$, $2<q<2+\frac{4}{N}$ and $\mu,a>0$.  Then for $\mu a^{q-q\gamma_q}<\alpha_{N,q}$,
\begin{eqnarray}\label{eq0068}
m_{a,\mu}^-=\inf_{u\in\mathcal{P}_-^{a,\mu}}\mathcal{E}_\mu(u)<m_{a,\mu}^++\frac{1}{3}S^{\frac{3}{2}}.
\end{eqnarray}
\end{lemma}
\begin{proof}
Since $N=3$, we have $U_\ve=3^{\frac14}(\frac{\ve}{\ve^2+|x|^2})^\frac12$.  Let $W_\ve =\chi (x) U_\epsilon $ where $ \chi (x)$ is a cut-off function such that $ \chi (x)=1 $ for $|x| \leq 1$ and $ \chi (x)=0$ for $|x|>2$.  By simple computations, we have that
\begin{eqnarray}\label{eqnewnew9997}
\|\nabla W_\ve\|_2^2=S^{\frac32}+O(\ve),\quad \|W_\ve\|_6^6=S^{\frac{3}{2}}+O(\ve^3)
\end{eqnarray}
and
\begin{eqnarray}\label{eqnewnew9994}
\|W_\ve\|_p^p\sim\left\{\aligned &\ve^{3-\frac p2},\quad 3<p<6;\\
&\ve^{\frac32}\ln\frac{1}{\ve},\quad p=3;\\
&\ve^{\frac{p}{2}},\quad 2\leq p<3.
\endaligned\right.
\end{eqnarray}
Now, we define $\widehat{W}_{\ve,t}=u_{\mu,+}+tW_\ve$ and $\overline{W}_{\ve,t}=s^{\frac12}\widehat{W}_{\ve,t}(sx)$.  Then it is well-known that
\begin{eqnarray}\label{eqnewnew9992}
\|\nabla \overline{W}_{\ve,t}\|_2^2=\|\nabla \widehat{W}_{\ve,t}\|_2^2,\quad\|\overline{W}_{\ve,t}\|_6^6=\|\widehat{W}_{\ve,t}\|_6^6,
\end{eqnarray}
and
\begin{eqnarray}\label{eqnewnew9991}
\|\overline{W}_{\ve,t}\|_2^2=s^{-2}\|\widehat{W}_{\ve,t}\|_2^2,\quad\|\overline{W}_{\ve,t}\|_q^q=s^{q\gamma_q-q}\|\widehat{W}_{\ve,t}\|_q^q.
\end{eqnarray}
We choose $s=\frac{\|\widehat{W}_{\ve,t}\|_2}{a}$, then $\overline{W}_{\ve,t}\in\mathcal{S}_a$.  By \cite[Lemma~4.2]{S20}, there exist $\tau_{\ve,t}>0$ such that
$(\overline{W}_{\ve,t})_{\tau_{\ve,t}}\in\mathcal{P}_-^{a,\mu}$, where $(\overline{W}_{\ve,t})_{\tau_{\ve,t}}=\tau_{\ve,t}^{\frac{3}{2}}\overline{W}_{\ve,t}(\tau_{\ve,t}x)$.  Thus,
\begin{eqnarray}\label{eqnewnew9993}
\|\nabla \overline{W}_{\ve,t}\|_2^2\tau_{\ve,t}^{2-q\gamma_q}=\mu\gamma_q\|\overline{W}_{\ve,t}\|_q^q+\|\overline{W}_{\ve,t}\|_{2^*}^{2^*}\tau_{\ve,t}^{2^*-q\gamma_q}.
\end{eqnarray}
Since $u_{\mu,+}\in\mathcal{P}_+^{a,\mu}$, by \cite[Lemma~4.2]{S20}, $\tau_{\ve,0}>1$.  By \eqref{eqnewnew9997} and \eqref{eqnewnew9993}, we also know that $\tau_{\ve,t}\to0$ as $t\to+\infty$ uniformly for $\ve>0$ sufficiently small.  Since $\tau_{\ve,t}$ is unique by \cite[Lemma~4.2]{S20}, it is standard to show that $\tau_{\ve,t}$ is continuous for $t$, which implies that there exists $t_\ve>0$ such that $\tau_{\ve, t_\ve}=1$.  It follows that
\begin{eqnarray}\label{eqnewnew9990}
m_{\mu,a}^-\leq\sup_{t\geq0}\mathcal{E}_\mu(\overline{W}_{\ve,t}).
\end{eqnarray}
Recall that $u_{\mu,+}\in\mathcal{S}_a$ and $W_\ve$ are positive, by \eqref{eqnewnew9997}, \eqref{eqnewnew9992} and \eqref{eqnewnew9991}, there exists $t_0>0$ such that
\begin{eqnarray}\label{eqnewnew9989}
\mathcal{E}_\mu(\overline{W}_{\ve,t})=(\frac{1}{2}\|\nabla \widehat{W}_{\ve,t}\|_2^2-\frac{\mu}{q}s^{q\gamma_q-q}\|\widehat{W}_{\ve,t}\|_q^q-\frac{1}{6}\|\widehat{W}_{\ve,t}\|_6^6)<m_{\mu,a}^++\frac{1}{3}S^{\frac{3}{2}}-\sigma'
\end{eqnarray}
for $t<\frac{1}{t_0}$ and $t>t_0$ with $\sigma'>0$.  Since $u_{\mu,+}$ is radial solution of \eqref{eq0001} and exponentially decays to zero as $r\to+\infty$,
\begin{eqnarray*}
\int_{\bbr^3}u_{\mu,+}W_\ve\sim\ve^{\frac52}\int_1^{\frac{1}{\ve}}(\frac{1}{1+r^2})^{\frac12}r^2\sim\ve^{\frac12}
\end{eqnarray*}
and
\begin{eqnarray}
\int_{\bbr^3}u_{\mu,+}W_\ve^5\sim\ve^{\frac12}\int_1^{\frac{1}{\ve}}(\frac{1}{1+r^2})^{\frac52}r^2\sim\ve^{\frac12}.\label{eqnewnew9977}
\end{eqnarray}
Thus, by \eqref{eqnewnew9994},
\begin{eqnarray*}
s^2=\frac{\|\widehat{W}_{\ve,t}\|_2^2}{a^2}=1+\frac{2t}{a^2}\int_{\bbr^3}u_{\mu,+}W_\ve+t^2\|W_\ve\|_2^2=1+O(\ve^{\frac12})
\end{eqnarray*}
for $t_0^{-1}\leq t\leq t_0$.  Since it is easy to see that $f(t)=(1+t)^{q}-1-t^q-qt-qt^{q-1}\geq0$ for all $t\geq0$ in the case of $q\geq3$,
by \eqref{eqnewnew9992}, \eqref{eqnewnew9991} and the fact that $u_{\mu,+}$ is a solution of \eqref{eq0001} for some $\lambda_{\mu,+}<0$,
\begin{eqnarray*}
\mathcal{E}_\mu(\overline{W}_{\ve,t})&=&\frac{1}{2}\|\nabla \widehat{W}_{\ve,t}\|_2^2-\frac{\mu}{q}s^{q\gamma_q-q}\|\widehat{W}_{\ve,t}\|_q^q-\frac{1}{6}\|\widehat{W}_{\ve,t}\|_6^6\notag\\
&\leq& m_{\mu,a}^++\mathcal{E}_\mu(tW_\ve)-\int_{\bbr^3}(tW_\ve)^{5}u_{\mu,+}\\
&&+t(\lambda_{\mu,+}\int_{\bbr^3}u_{\mu,+}W_\ve+\frac{\mu}{a^2}(\gamma_q-1)\|\widehat{W}_{\ve,t}\|_q^q\int_{\bbr^3}u_{\mu,+}W_\ve)+O(\ve)\\
&=& m_{\mu,a}^++\mathcal{E}_\mu(tW_\ve)-\int_{\bbr^3}(tW_\ve)^{5}u_{\mu,+}+O(\ve)
\end{eqnarray*}
for $t_0^{-1}\leq t\leq t_0$,
where we have used the fact that $\lambda_{\mu,+}a^2=\lambda_{\mu,+}\|u_{\mu,+}\|_2^2=\mu(\gamma_q-1)\|u_{\mu,+}\|_q^q$ which comes from the Pohozaev identity satisfied by $u_{\mu,+}$.
Now, for $t_0^{-1}\leq t\leq t_0$,
by \eqref{eqnewnew9997}, \eqref{eqnewnew9994} and \eqref{eqnewnew9977},
\begin{eqnarray*}
\mathcal{E}_\mu(\overline{W}_{\ve,t})\leq m_{\mu,a}^++\frac{1}{3}S^{\frac{3}{2}}+O(\ve)-C\ve^{\frac12}<m_{\mu,a}^++\frac{1}{3}S^{\frac{3}{2}}
\end{eqnarray*}
by taking $\ve>0$ sufficiently small.  It follows from \eqref{eqnewnew9989} that
\begin{eqnarray}
\sup_{t\geq0}\mathcal{E}_\mu(\overline{W}_{\ve,t})<m_{\mu,a}^++\frac{1}{3}S^{\frac{3}{2}}.
\end{eqnarray}
The conclusion then follows from \eqref{eqnewnew9990}.
\end{proof}

\begin{remark}
It is worth pointing our that the above argument also works for $N\geq4$.  In these cases, we have
\begin{eqnarray*}
\|\nabla W_\ve\|_2^2=S^{\frac N2} +O(\ve^{N-2}),\quad \|W_\ve\|_6^6=S^{\frac{N}{2}} +O(\ve^N)
\end{eqnarray*}
and
\begin{eqnarray*}
\|W_\ve\|_q^q \sim \ve^{N-\frac{(N-2)q}{2}},\quad \|W_\ve\|_2^2\sim\left\{\aligned &\ve^2\ln\frac{1}{\ve},\quad N=4,\\
&\ve^2,\quad N\geq5.\endaligned\right.
\end{eqnarray*}
Moreover, similar to \eqref{eqnewnew9977},
\begin{eqnarray*}
\int_{\bbr^N}u_{\mu,+}^pW_\ve\sim\ve^{\frac{N-2}{2}}\quad\text{for all }p\geq1.
\end{eqnarray*}
It follows that
\begin{eqnarray*}
\mathcal{E}_\mu(\overline{W}_{\ve,t})&=&\frac{1}{2}\|\nabla \widehat{W}_{\ve,t}\|_2^2-\frac{\mu}{q}s^{q\gamma_q-q}\|\widehat{W}_{\ve,t}\|_q^q-\frac{1}{6}\|\widehat{W}_{\ve,t}\|_6^6\notag\\
&\leq& m_{\mu,a}^++\mathcal{E}_\mu(tW_\ve)\\
&&+t(\lambda_{\mu,+}\int_{\bbr^N}u_{\mu,+}W_\ve+\frac{\mu}{a^2}(\gamma_q-1)\|\widehat{W}_{\ve,t}\|_q^q\int_{\bbr^N}u_{\mu,+}W_\ve)+O(\ve^{2}\ln\frac{1}{\ve})\\
&=& m_{\mu,a}^++\mathcal{E}_\mu(tW_\ve)+O(\ve^{N-2})\\
&\leq&m_{\mu,a}^++\frac{1}{N}S^{\frac{N}{2}}-C\ve^{N-\frac{(N-2)q}{2}}+O(\ve^{2}\ln\frac{1}{\ve})\\
&<&m_{\mu,a}^++\frac{1}{N}S^{\frac{N}{2}}
\end{eqnarray*}
for $t_0^{-1}\leq t\leq t_0$ by taking $\ve>0$ sufficiently small since $N\geq4$ and $q>2$.  Our proof is slightly simpler than that of \cite{JL201} since our test function is radial and we do not need other variational formulas of $m_{\mu,a}^-$.
\end{remark}

For every $c>0$ such that $\mu c^{q-q\gamma_q}<\alpha_{N,q}$, let $u\in\mathcal{P}_\pm^{c,\mu}$, then $v_b=\frac{b}{c}u\in\mathcal{S}_b$ for all $b>0$.  By \cite[Lemma~4.2]{S20}, there exists $\tau_\pm(b)>0$ such that
\begin{eqnarray*}
(v_b)_{\tau_\pm(b)}=(\tau_\pm(b))^{\frac{N}{2}}v_b(\tau_\pm(b)x)\in\mathcal{P}_\pm^{b,\mu},
\end{eqnarray*}
where $b>0$ such that $\mu b^{q-q\gamma_q}<\alpha_{N,q}$.
Clearly, $\tau_\pm(c)=1$.
\begin{lemma}\label{lemma0003}
Let $2<q<2+\frac{4}{N}$.  For every $c>0$ such that $\mu c^{q-q\gamma_q}<\alpha_{N,q}$, $\tau'_\pm(c)$ exist and
\begin{eqnarray}\label{eq0009}
\tau'_\pm(c)=\frac{\mu q\gamma_q\|u\|_q^q+2^*\|u\|_{2^*}^{2^*}-2\|\nabla u\|_2^2}{c(2\|\nabla u\|_2^2-\mu q\gamma_q^2\|u\|_q^q-2^*\|u\|_{2^*}^{2^*})}.
\end{eqnarray}
Moreover, $\mathcal{E}_\mu((v_b)_{\tau_\pm(b)})<\mathcal{E}_\mu(u)$ for all $b>c$ such that $\mu b^{q-q\gamma_q}<\alpha_{N,q}$.
\end{lemma}
\begin{proof}
The proof is mainly inspired by \cite{CZ121}.  Since $(v_b)_{\tau_\pm(b)}\in\mathcal{P}_\pm^{b,\mu}$, we have
\begin{eqnarray*}
(\frac{b}{c}\tau(b))^2\|\nabla u\|_2^2=(\frac{b}{c})^{q}(\tau(b))^{q\gamma_q}\mu\gamma_q\|u\|_q^q+(\frac{b}{c}\tau(b))^{2^*}\|u\|_{2^*}^{2^*}.
\end{eqnarray*}
Now, if we define the function
\begin{eqnarray*}
\Phi(b,\tau)=(\frac{b\tau}{c})^2\|\nabla u\|_2^2-(\frac{b}{c})^{q}\tau^{q\gamma_q}\mu\gamma_q\|u\|_q^q-(\frac{b\tau}{c})^{2^*}\|u\|_{2^*}^{2^*},
\end{eqnarray*}
then $\Phi(b,\tau(b))\equiv0$ for $b>0$ such that $\mu b^{q-q\gamma_q}<\alpha_{N,q}$.  Since $u\in\mathcal{P}_\pm^{c,\mu}$,
\begin{eqnarray*}
\partial_{\tau}\Phi(c,1)=2\|\nabla u\|_2^2-\mu q\gamma_q^2\|u\|_q^q-2^*\|u\|_{2^*}^{2^*}\not=0.
\end{eqnarray*}
It follows from the implicit function theorem that $\tau'_\pm(c)$ exist and \eqref{eq0009} holds.  By \eqref{eq0008} and $q<2^*$, $1-\gamma_q>0$.  Thus, by $u\in\mathcal{P}_\pm^{c,\mu}$,
\begin{eqnarray*}
1+c\tau'(c)&=&1+\frac{\mu q\gamma_q\|u\|_q^q+2^*\|u\|_{2^*}^{2^*}-2\|\nabla u\|_2^2}{2\|\nabla u\|_2^2-\mu q\gamma_q^2\|u\|_q^q-2^*\|u\|_{2^*}^{2^*}}\\
&=&\frac{\mu q\gamma_q(1-\gamma_q)\|u\|_q^q}{2\|\nabla u\|_2^2-\mu q\gamma_q^2\|u\|_q^q-2^*\|u\|_{2^*}^{2^*}}.
\end{eqnarray*}
Since $(v_b)_{\tau_\pm(b)}\in\mathcal{P}_\pm^{b,\mu}$ and $u\in\mathcal{P}_\pm^{c,\mu}$,
\begin{eqnarray*}
\mathcal{E}_\mu((v_b)_{\tau_\pm(b)})&=&(\frac{1}{2}-\frac{1}{q\gamma_q})\|\nabla (v_b)_{\tau(b)}\|_2^2+(\frac{1}{q\gamma_q}-\frac{1}{2^*})\|(v_b)_{\tau(b)}\|_{2^*}^{2^*}\\
&=&(\frac{b}{c}\tau(b))^2(\frac{1}{2}-\frac{1}{q\gamma_q})\|\nabla u\|_2^2+(\frac{b}{c}\tau(b))^{2^*}(\frac{1}{q\gamma_q}-\frac{1}{2^*})\|u\|_{2^*}^{2^*}\\
&=&(\frac{1}{2}-\frac{1}{q\gamma_q})\|\nabla u\|_2^2+(\frac{1}{q\gamma_q}-\frac{1}{2^*})\|u\|_{2^*}^{2^*}+o(b-c)\\
&&+\frac{1+c\tau'(c)}{c}(2(\frac{1}{2}-\frac{1}{q\gamma_q})\|\nabla u\|_2^2+2^*(\frac{1}{q\gamma_q}-\frac{1}{2^*})\|u\|_{2^*}^{2^*})(b-c)\\
&=&\mathcal{E}_\mu(u)-\frac{\mu(1-\gamma_q)\|u\|_q^q}{c}(b-c)\\
&&+o(b-c).
\end{eqnarray*}
Therefore,
\begin{eqnarray*}
\frac{d\mathcal{E}_\mu((v_b)_{\tau_\pm(b)})}{db}|_{b=c}=-\frac{\mu(1-\gamma_q)\|u\|_q^q}{c}<0.
\end{eqnarray*}
Since $c>0$, which satisfies $\mu c^{q-q\gamma_q}<\alpha_{N,q}$, is arbitrary and $(v_b)_{\tau_\pm(b)}\in \mathcal{P}_\pm^{b,\mu}$, we have $\mathcal{E}_\mu((v_b)_{\tau_\pm(b)})<\mathcal{E}_\mu(u)$ for all $b>c$ such that $\mu b^{q-q\gamma_q}<\alpha_{N,q}$.
\end{proof}

With Lemma~\ref{lemma0003} in hands, we can obtain the following.
\begin{proposition}\label{prop0001}
Let $2<q<2+\frac{4}{N}$ and $\mu a^{q-q\gamma_q}<\alpha_{N,q}$.  If $m_{a,\mu}^-<m_{a,\mu}^++\frac{1}{N}S^{\frac{N}{2}}$
then
\begin{eqnarray*}
m_{a,\mu}^-=\inf_{u\in\mathcal{P}_-^{a,\mu}}\mathcal{E}_\mu(u)
\end{eqnarray*}
can be attained by some $u_{a,\mu,-}$ which is real valued, positive, radially symmetric and decreasing in $r=|x|$.  Moreover, \eqref{eq0001} has a second solution $u_{a,\mu,-}$ with some $\lambda_{a,\mu,-}<0$ which is real valued, positive, radially symmetric and radially decreasing.
\end{proposition}
\begin{proof}
Let $\{u_n\}\subset\mathcal{P}_-^{a,\mu}$ be a minimizing sequence.  Then by taking $|u_n|$ and adapting the Schwarz symmetrization to $|u_n|$ if necessary, we can obtain a new minimizing sequence, say $\{u_n\}$ again, such that $u_n$ are all real valued, nonnegative, radially symmetric and decreasing in $r=|x|$.  Since $\{u_n\}\subset\mathcal{P}_-^{a,\mu}$, we have
\begin{eqnarray}\label{eq0011}
\mathcal{E}_\mu(u_n)=\frac{\mu}{q}(\frac{q\gamma_q}{2}-1)\|u_n\|_q^q+\frac{1}{N}\|u_n\|_{2^*}^{2^*}.
\end{eqnarray}
Thus, by the H\"older and Young inequalities and $\{u_n\}\subset\mathcal{P}_-^{a,\mu}$ again, we known that $\{u_n\}$ is bounded in $H^1(\bbr^N)$ and thus, $u_n\rightharpoonup u_0$ weakly in $H^1(\bbr^N)$ as $n\to\infty$ up to a subsequence.  Since $u_n$ are all radial, by Struss's radial lemma (cf. \cite[Lemma A.IV, Theorem A.I']{BL83} or \cite[Lemma~3.1]{MM14}) and the Sobolev embedding theorem, $u_n\to u_0$ strongly in $L^q(\bbr^N)$ as $n\to\infty$ up to a subsequence.  Without loss of generality, we assume that $u_n\rightharpoonup u_0$ weakly in $H^1(\bbr^N)$ and $u_n\to u_0$ strongly in $L^q(\bbr^N)$ as $n\to\infty$.  We claim that $u_0\not=0$.  If not, then $u_n\to 0$ strongly in $L^q(\bbr^N)$ as $n\to\infty$.  It follows from $\{u_n\}\subset\mathcal{P}_-^{a,\mu}$ that
\begin{eqnarray*}
\|\nabla u_n\|_2^2=\|u_n\|_{2^*}^{2^*}+o_n(1),
\end{eqnarray*}
which together with the Sobolev inequality~\eqref{eq0048}, implies that either $u_n\to0$ strongly in $D^{1,2}(\bbr^N)$ as $n\to\infty$ or $\|\nabla u_n\|_2^2=\|u_n\|_{2^*}^{2^*}+o_n(1)\geq S^{\frac{N}{2}}+o_n(1)$.  Hence, by \eqref{eq0011}, either $m_{a,\mu}^-=0$ or $m_{a,\mu}^-\geq\frac{1}{N}S^{\frac{N}{2}}$, which contradicts $\mathcal{E}_\mu(u)\gtrsim1$ for $u\in\mathcal{P}_-^{c,\mu}$ and Lemma~\ref{lemma0002}.  We remark that $\mathcal{E}_\mu(u)\gtrsim1$ for $u\in\mathcal{P}_-^{c,\mu}$ comes from similar arguments as used for \cite[Lemma~5.7]{S201}.  Therefore, we must have $u_0\not=0$.  Let $v_n=u_n-u_0$.
Then there are two cases:
\begin{enumerate}
\item[$(i)$]\quad $v_n\to0$ strongly in $H^1(\bbr^N)$ as $n\to\infty$ up to a subsequence.
\item[$(ii)$]\quad $\|\nabla v_n\|_2^2+\|v_n\|_2^2\gtrsim1$.
\end{enumerate}
In the case~$(i)$, $u_0\in\mathcal{P}_-^{a,\mu}$ and $m_{a,\mu}^-$ is attained by $u_0$ which is real valued, radially symmetric, nonnegative and decreasing in $r=|x|$.  By \cite[Proposition~1.5]{S20}, $u_0$ is a solution of \eqref{eq0001} with some $\lambda_0\in\bbr$ which appears as a Lagrange multiplier.  By multiplying \eqref{eq0001} with $u_0$ and integrating by parts, and using $u_0\in\mathcal{P}_-^{a,\mu}$, we have
\begin{eqnarray*}
\lambda_0a^2=\mu(\gamma_q-1)\|u_0\|_q^q<0,
\end{eqnarray*}
which implies $\lambda_0<0$.  Now, by the maximum principle and classical elliptic estimates, we know that $u_0$ is positive.  It remains to consider the case~$(ii)$.  Let $\|u_0\|_2^2=t_0^2$, then by the Fatou lemma, $0<t_0\leq a$.  There are two subcases:
\begin{enumerate}
\item[$(ii_1)$]\quad $\|v_n\|_{2^*}\to0$ as $n\to\infty$ up to a subsequence.
\item[$(ii_2)$]\quad $\|v_n\|_{2^*}^{2^*}\gtrsim1$.
\end{enumerate}
In the subcase~$(ii_1)$, by \cite[Lemma~4.2]{S20}, there exists $s_0>0$ such that $(u_0)_{s_0}\in\mathcal{P}_-^{t_0,\mu}$.  By \cite[Lemma~4.2]{S20} once more, $\{u_n\}\subset\mathcal{P}_-^{a,\mu}$ and $u_n\to u_0$ strongly in $L^{2^*}(\bbr^N)\cap L^q(\bbr^N)$ as $n\to\infty$ up to a subsequence,
\begin{eqnarray*}
m_{a,\mu}^-+o_n(1)=\mathcal{E}_\mu(u_n)\geq\mathcal{E}_\mu((u_n)_{s_0})\geq\mathcal{E}_\mu((u_0)_{s_0})+o_n(1).
\end{eqnarray*}
By Lemma~\ref{lemma0003}, we have $m_{t_0,\mu}^-\geq m_{a,\mu}^-$.  Thus, $\mathcal{E}_\mu((u_0)_{s_0})=m_{t_0,\mu}^-$ and $m_{t_0,\mu}^-= m_{a,\mu}^-$.  If $t_0<a$ then by taking $(u_0)_{s_0}$ as the test function in the proof of Lemma~\ref{lemma0003}, we know that $m_{t_0,\mu}^->m_{a,\mu}^-$, which is a contradiction.  Thus, in the subcase~$(ii_1)$, we must have $t_0=a$ and so that $m_{a,\mu}^-$ is attained by $(u_0)_{s_0}$ which is real valued, radially symmetric, nonnegative and decreasing in $r=|x|$.  As above, we can show that $(u_0)_{s_0}$ is positive and $(u_0)_{s_0}$ is a solution of \eqref{eq0001} with some $\lambda_0'<0$.  It remains to consider the subcase~$(ii_2)$.  Let
\begin{eqnarray*}
s_n=\bigg(\frac{\|\nabla v_n\|_2^2}{\|v_n\|_{2^*}^{2^*}}\bigg)^{\frac{1}{2^*-2}}.
\end{eqnarray*}
Then in the subcase~$(ii_2)$, $s_n\lesssim1$ and by the Sobolev inequality~\eqref{eq0048},
\begin{eqnarray*}
\|\nabla(v_n)_{s_n}\|_2^2=\|(v_n)_{s_n}\|_{2^*}^{2^*}\geq\frac{1}{N}S^{\frac{N}{2}}.
\end{eqnarray*}
Since $0<t_0\leq a$, by \cite[Lemma~4.2]{S20}, there exists $\tau_0>0$ such that $(u_0)_{\tau_0}\in\mathcal{P}_-^{t_0,\mu}$.  We claim that $s_n\geq\tau_0$ up to a subsequence.  Suppose the contrary that $s_n<\tau_0$ for all $n$.  Then by \cite[Lemma~4.2]{S20} once more, the Brezis-Lieb lemma (cf. \cite[Lemma~1.32]{W96}), Lemma~\ref{lemma0003}, the fact that $u_n\to u_0$ strongly in $L^q(\bbr^N)$ as $n\to\infty$ and the boundedness of $\{s_n\}$,
\begin{eqnarray*}
m_{a,\mu}^-+o_n(1)&=&\mathcal{E}_\mu(u_n)\\
&\geq&\mathcal{E}_\mu((u_n)_{s_n})\\
&=&\mathcal{E}_\mu((u_0)_{s_n})+\mathcal{E}_0((v_n)_{s_n})+o_n(1)\\
&\geq&m_{t_0,\mu}^++\frac{1}{N}S^{\frac{N}{2}}+o_n(1)\\
&\geq&m_{a,\mu}^++\frac{1}{N}S^{\frac{N}{2}}+o_n(1),
\end{eqnarray*}
which is impossible.  Thus, we must have $s_n\geq\tau_0$ up to a subsequence.  Without loss of generality, we may assume that $s_n\geq\tau_0$ for all $n\in\bbn$.  Again, by \cite[Lemma~4.2]{S20}, the Brezis-Lieb lemma (cf. \cite[Lemma~1.32]{W96}) and the fact that $u_n\to u_0$ strongly in $L^q(\bbr^N)$ as $n\to\infty$,
\begin{eqnarray*}
m_{a,\mu}^-+o_n(1)=\mathcal{E}_\mu(u_n)\geq\mathcal{E}_\mu((u_n)_{\tau_0})=\mathcal{E}_\mu((u_0)_{\tau_0})+\mathcal{E}_0((v_n)_{\tau_0})+o_n(1).
\end{eqnarray*}
Since $s_n\geq\tau_0$, by \cite[Proposition~2.2]{S20}, $\mathcal{E}_0((v_n)_{\tau_0})\geq0$, which, together with Lemma~\ref{lemma0003}, implies that $t_0=a$ and $m_{a,\mu}^-$ is attained by $(u_0)_{\tau_0}$.  Clearly,  $(u_0)_{\tau_0}$ is real valued, radially symmetric, nonnegative and decreasing in $r=|x|$.  As above, we can show that $(u_0)_{\tau_0}$ is positive and $(u_0)_{\tau_0}$ is a solution of \eqref{eq0001} with some $\lambda_0''<0$.  Therefore, we have proved that $m_{a,\mu}^-$ can always be attained by some $u_{a,\mu,-}$ which is real valued, radially symmetric, positive and decreasing in $r=|x|$.  By \cite[Proposition~1.5]{S20}, \eqref{eq0001} has a second solution $u_{a,\mu,-}$ which is real valued, radially symmetric, positive and decreasing in $r=|x|$.
\end{proof}

Our next goal in this section is to prove the existence and nonexistence of ground states for $\mu a^{q-q\gamma_q}\geq\alpha_{N,q}$ in the $L^2$-critical and supercritical cases, which gives partial answers to the question~$(Q_2)$.
In these two cases, $2+\frac{4}{N}\leq q<2^*$, which implies
\begin{eqnarray*}\label{eq0058}
q\gamma_q\geq2.
\end{eqnarray*}
We recall that the constant $\alpha_{N,q}$ is given by \cite[Theorem~1.1]{S20}.  For $q=2+\frac{4}{N}$, by \cite[(5,1)]{S20},
\begin{eqnarray}\label{eq1101}
\alpha_{N,q}=C_{N,q}^{-q}(1+\frac{2}{N})=\frac{1}{C_{N,q}^q\gamma_q},
\end{eqnarray}
where $C_{N,q}$ is the optimal constant in the Gagliardo--Nirenberg inequality~\eqref{eq0059}.
\begin{lemma}\label{lemma0005}
Let $N\geq3$ and $2+\frac{4}{N}\leq q<2^*$.  Then $m_{a,\mu}^-$ is strictly decreasing for $0<\mu<a^{q\gamma_q-q}\alpha_{N,q}$ and is nonincreasing for $\mu\geq a^{q\gamma_q-q}\alpha_{N,q}$, where $m_{a,\mu}^-$ is given by \eqref{eq0045}.  Moreover, $0<m_{a,\mu}^-<\frac{1}{N}S^{\frac{N}{2}}$ for all $\mu>0$ in the case of $2+\frac{4}{N}< q<2^*$ while, $m_{a,\mu}^-=0$ for $\mu\geq a^{q\gamma_q-q}\alpha_{N,q}$ in the case of $q=2+\frac{4}{N}$.
\end{lemma}
\begin{proof}
Modified the proof of \cite[Lemma~8.2]{S20} in a trivial way (or by Lemma~\ref{lemma0003} and \cite[Theorem~1.1]{S20}), we can show that $m_{a,\mu}^-$ is strictly decreasing for $0<\mu<a^{q\gamma_q-q}\alpha_{N,q}$.
For $\mu\geq a^{q\gamma_q-q}\alpha_{N,q}$, let us consider the fibering map
\begin{eqnarray*}
\Psi_u(t)=\frac{t^2}{2}\|\nabla u\|_2^2-\frac{\mu t^{q\gamma_q}}{q}\|u\|_q^q-\frac{t^{2^*}}{2^*}\|u\|_{2^*}^{2^*},
\end{eqnarray*}
as that in \cite{S20}.  For $2+\frac{4}{N}< q<2^*$, it has been proved in \cite[Lemma~6.1]{S20} that for every $u\in\mathcal{S}_a$, there exists $t_u>0$ such that
$\Psi_u(t)$ is strictly increasing in $(0, t_u)$, is strictly decreasing in $(t_u, +\infty)$ and
\begin{eqnarray*}
(u)_{t_u}=t_u^{\frac{N}{2}}u(t_ux)\in\mathcal{P}_-^{a, \mu}.
\end{eqnarray*}
Moreover, by \cite[Lemma~6.2]{S20}, we have $m_{a,\mu}^->0$ for all $\mu>0$ in the $L^2$-supercritical case $2+\frac{4}{N}<q<2^*$.  It follows that we can always choose $v_\ve\in\mathcal{P}_-^{a,\mu}$ such that $\mathcal{E}_\mu(v_\ve)<m_{a,\mu}^-+\ve$ in the $L^2$-supercritical case $2+\frac{4}{N}<q<2^*$.  Then by similar arguments as used for \cite[Lemma~8.2]{S20} (or by Lemma~\ref{lemma0003}), we have
\begin{eqnarray*}
m_{a,\mu'}^-<m_{a,\mu}^-+\ve\quad\text{for all }\mu'>\mu.
\end{eqnarray*}
Since $\ve>0$ and $\mu\geq a^{q\gamma_q-q}\alpha_{N,q}$ are arbitrary, $m_{a,\mu}^-$ is nonincreasing for $\mu\geq a^{q\gamma_q-q}\alpha_{N,q}$ in the $L^2$-supercritical case $2+\frac{4}{N}<q<2^*$.  It follows from \cite[Lemma~6.4]{S20} that $m_{a,\mu}^-<\frac{1}{N}S^{\frac{N}{2}}$ for all $\mu>0$.

In the $L^2$-critical case $q=2+\frac{4}{N}$, since
\begin{eqnarray*}
\sup_{u\in\mathcal{S}_a}\frac{\|\nabla u\|_2}{\|u\|_q}=+\infty.
\end{eqnarray*}
For all $\mu>0$, we can always choose $u\in\mathcal{S}_a$ such that $\frac{\|\nabla u\|_2}{\|u\|_q}>\mu\gamma_q$.  Indeed, if $\sup_{u\in\mathcal{S}_a}\frac{\|\nabla u\|_2}{\|u\|_q}\lesssim1$, then by the Gagliardo-Nirenberg inequality,
\begin{eqnarray*}
\|\nabla u\|_2\lesssim\|u\|_q\lesssim\|\nabla u\|_2^{\gamma_q}\quad\text{for all }u\in\mathcal{S}_a,
\end{eqnarray*}
which implies
\begin{eqnarray*}
\sup_{u\in\mathcal{S}_a}\|\nabla u\|_2\lesssim1.
\end{eqnarray*}
It is impossible since in any ball $B_R(0)$, the eigenvalue problem $-\Delta u=\lambda u$, with Dirichlet boundary conditions, has a sequence of eigenvalues $\lambda_j\to+\infty$ as $j\to\infty$.
We note that $q\gamma_q=2$ in the $L^2$-critical case $q=2+\frac{4}{N}$.  Thus,
\begin{eqnarray*}
\Psi_u'(t)=(\|\nabla u\|_2^2-\mu\gamma_q\|u\|_q^q)t-t^{2^*-1}\|u\|_{2^*}^{2^*}=0
\end{eqnarray*}
has a unique solution $t_u>0$ for $u\in\mathcal{S}_a$ such that $\frac{\|\nabla u\|_2}{\|u\|_q}>\mu\gamma_q$.  Moreover, $\Psi_u(t)$ is strictly increasing in $(0, t_u)$, is strictly decreasing in $(t_u, +\infty)$ and
\begin{eqnarray*}
(u)_{t_u}=t_u^{\frac{N}{2}}u(t_ux)\in\mathcal{P}_{a, \mu}=\mathcal{P}_-^{a, \mu}.
\end{eqnarray*}
Thus, $\mathcal{P}_{a, \mu}=\mathcal{P}_-^{a, \mu}\not=\emptyset$ and $\Psi_u(t_u)=\max_{t\geq0}\Psi_u(t)$
for all $\mu>0$ and all $u\in\mathcal{S}_a$ such that $\frac{\|\nabla u\|_2}{\|u\|_q}>\mu\gamma_q$.  Now, as in the $L^2$-supercritical case $2+\frac{4}{N}<q<2^*$, by similar arguments as used for \cite[Lemma~8.2]{S20}, we can show that $m_{a,\mu}^-$ is nonincreasing for $\mu\geq a^{q\gamma_q-q}\alpha_{N,q}$ in the $L^2$-critical case $q=2+\frac{4}{N}$.  It remains to prove that $m_{a,\mu}^-=0$ for $\mu\geq a^{q\gamma_q-q}\alpha_{N,q}$ in the case of $q=2+\frac{4}{N}$.  Let $\{\varphi_n\}$ be the minimizing sequence of the Gagliardo-Nirenberg inequality~\eqref{eq0059}.  Then by scaling $\frac{at_n^{\frac{N}{2}}}{\|\varphi_n\|_2}\varphi_n(t_nx)$ if necessary, we may assume that $\|\varphi_n\|_2^2=a^2$, $\|\varphi_n\|_q^q=1$ and $\|\nabla \varphi_n\|_2^2=C_{N,q}^{-\frac{2}{\gamma_q}}a^{\frac{2(\gamma_q-1)}{\gamma_q}}+o_n(1)$.
Let us consider the following function:
\begin{eqnarray*}
h_{\varphi_n}(\mu,t)&=&t^2(\|\nabla \varphi_n\|_2^2-\mu\gamma_q\|\varphi_n\|_q^q)-t^{2^*}\|\varphi_n\|_{2^*}^{2^*}\\
&=&(C_{N,q}^{-\frac{2}{\gamma_q}}a^{\frac{2(\gamma_q-1)}{\gamma_q}}+o_n(1)-\mu\gamma_q)t^2-t^{2^*}\|\psi\|_{2^*}^{2^*}\\
&=&\gamma_q(\alpha_{N,q}a^{q\gamma_q-q}+o_n(1)-\mu)t^2-t^{2^*}\|\psi\|_{2^*}^{2^*},
\end{eqnarray*}
where we have used \eqref{eq1101}.
By \cite[Lemma~5.1]{S20}, there exists a unique $t_n(\mu)>0$ such that $h_{\varphi_n}(\mu,t_n(\mu))=0$ for $0<\mu<a^{q\gamma_q-q}\alpha_{N,q}$.  Thus, $(\varphi_n)_{t_n(\mu)}\in\mathcal{P}_{a,\mu}$ for $0<\mu<a^{q\gamma_q-q}\alpha_{N,q}$, where $(\varphi_n)_{t_n(\mu)}=[t_n(\mu)]^{\frac{N}{2}}\varphi_n(t_n(\mu)x)$.  Moreover,
since $\|\varphi_n\|_q^q=1$, by the H\"older inequality, $\|\varphi_n\|_{2^*}\gtrsim1$.  It follows that $t(\mu)\to 0$ as $\mu\to a^{q\gamma_q-q}\alpha_{N,q}$, which implies
\begin{eqnarray*}
\mathcal{E}_\mu((\psi)_{t(\mu)})=\frac{1}{N}\|\varphi_n\|_{2^*}^{2^*}[t_n(\mu)]^{2^*}=o_n(1)
\end{eqnarray*}
as $\mu\to a^{q\gamma_q-q}\alpha_{N,q}$ in the $L^2$-critical case $q=2+\frac{4}{N}$.  Thus, we must have $m_{a,\mu}^-\leq0$ for $\mu=a^{q\gamma_q-q}\alpha_{N,q}$.  By the monotone property of $m_{a,\mu}^-$ stated in Lemma~\ref{lemma0003}, $m_{a,\mu}^-\leq0$ for $\mu\geq a^{q\gamma_q-q}\alpha_{N,q}$.
Recall that we always have
\begin{eqnarray}\label{eq0013}
\mathcal{E}_\mu(u)=\frac{1}{N}\|u\|_{2^*}^{2^*}\geq0\quad\text{for all }u\in\mathcal{P}_{a,\mu},
\end{eqnarray}
thus, we must have $m_{a,\mu}^-=0$ for $\mu\geq a^{q\gamma_q-q}\alpha_{N,q}$.
\end{proof}

With Lemma~\ref{lemma0005} in hands, we can obtain the following.
\begin{proposition}\label{prop0002}
Let $N\geq3$ and $2+\frac{4}{N}\leq q<2^*$.
\begin{enumerate}
\item[$(1)$]\quad If $2+\frac{4}{N}<q<2^*$, then $m_{a,\mu}^-$ is attained by same $u_{a, \mu,-}$ which is real valued, positive, radially symmetric and decreasing in $r=|x|$ for all $\mu>0$, and thus, $u_{a, \mu,-}$ is a solution of \eqref{eq0001} for all $\mu>0$ with some $\lambda_{a,\mu,-}<0$.
\item[$(2)$]\quad If $q=2+\frac{4}{N}$, then $m_{a,\mu}^-$ can not be attained and \eqref{eq0001} has no ground states for all $\mu\geq a^{q\gamma_q-q}\alpha_{N,q}$.
\end{enumerate}
\end{proposition}
\begin{proof}
$(1)$\quad By Lemma~\ref{lemma0005},  $0<m_{a,\mu}^-<\frac{1}{N}S^{\frac{N}{2}}$ for all $\mu>0$ in the case of $2+\frac{4}{N}< q<2^*$.  Now, by following the arguments in \cite[Section~6]{S20} step by step, we can show that $m_{a,\mu}^-$ is attained by some $u_{a, \mu,-}$ which is real valued, nonnegative, radially symmetric and decreasing in $r=|x|$ for all $\mu>0$ in the case of $2+\frac{4}{N}< q<2^*$.
By similar arguments as used for \cite[Lemma~6.2]{S201}, we know that $\mathcal{P}_{a,\mu}=\mathcal{P}_-^{a,\mu}\not=\emptyset$ is a natural constraint in $\mathcal{S}_a$ for all $\mu>0$ in the case of $2+\frac{4}{N}< q<2^*$.  Thus, $u_{a,\mu,-}$ is a solution of \eqref{eq0001} for all $\mu>0$ with some $\lambda_{a, \mu,-}$ in the case of $2+\frac{4}{N}< q<2^*$.  As that in the proof of Proposition~\ref{prop0001}, we can show that $\lambda_{a, \mu,-}<0$ and $u_{a,\mu,-}$ is positive.

$(2)$\quad Suppose the contrary that $m_{a,\mu}^-$ is attained by some $u_{a,\mu,-}$ for $\mu\geq a^{q\gamma_q-q}\alpha_{N,q}$, then by Lemma~\ref{lemma0005} and \eqref{eq0013}, $\|u_{a,\mu,-}\|_{2^*}^{2^*}=0$.  It is impossible since $u_{a,\mu,-}\in\mathcal{S}_a$.  Thus, $m_{a,\mu}^-$ can not be attained for $\mu\geq a^{q\gamma_q-q}\alpha_{N,q}$.  It follows that \eqref{eq0001} has no ground state for all $\mu\geq a^{q\gamma_q-q}\alpha_{N,q}$.
\end{proof}

\section{The asymptotic behavior of $u_{a,\mu,-}$}
In this section, we shall mainly study the question~$(Q_3)$ and give a precisely description of the asymptotic behavior of $u_{a,\mu,-}$ as $\mu\to0^+$.  Since we consider $\mu\to0^+$ now, the assumptions of \cite[Theorem~1.1]{S20}, \cite[Theorem~1.6]{JL201} and Proposition~\ref{prop0001} always hold and thus, $u_{a,\mu,-}$, which is a minimizer of $\mathcal{E}|_{\mathcal{S}_a}(u)$ on $\mathcal{P}_-^{a,\mu}$, exists for all $N\geq3$, $2<q<2^*$ for $\mu>0$ sufficiently small.
\begin{proposition}\label{prop0006}
Let $N\geq3$, $2<q<2^*$ and $u_{a,\mu,-}$ is a critical point of $\mathcal{E}|_{\mathcal{S}_a}(u)$ of mountain pass type.  If $N\geq5$ then $u_{a,\mu,-}\to U_{\ve_0}$ strongly in $H^1(\bbr^N)$ as $\mu\to0^+$,
where $U_{\ve_0}$ is the Aubin-Talanti babble such that $U_{\ve_0}\in\mathcal{S}_a$.  Moreover, if $N\geq9$, then up to translations and rotations, $u_{a,\mu,-}$ is the unique minimizer of $\mathcal{E}|_{\mathcal{S}_a}(u)$ on $\mathcal{P}_-^{a,\mu}$ for $\mu>0$ sufficiently small.
\end{proposition}
\begin{proof}
By \cite[Theorem~1.4]{S20}, $m_{a,\mu}^+\to0$ as $\mu\to0^+$.  Moreover, by \cite[Theorem~1.4]{S20} again, we know that
$m_{a,\mu}^-\to\frac{1}{N}S^{\frac{N}{2}}$ as $\mu\to0^+$, and
\begin{eqnarray*}
\|\nabla u_{a,\mu,-}\|_2^2,\|u_{a,\mu,-}\|_{2^*}^{2^*}\to S^{\frac{N}{2}}\quad\text{as }\mu\to0^+
\end{eqnarray*}
for $2+\frac{4}{N}\leq q<2^*$.  On the other hand, by \cite[Lemma~4.2]{S20} and similar arguments as used for \cite[Lemma~5.7]{S201}, we also have $m_{a,\mu}^-\gtrsim1$ for $\mu>0$ sufficiently small in the case of $2<q<2+\frac{4}{N}$.  Thus, by adapting similar arguments as used in the proof of \cite[Theorem~1.1]{S20} for the case of $2+\frac{4}{N}\leq q<2^*$ to the case of $2<q<2+\frac{4}{N}$, we can also show that
$m_{a,\mu}^-\to\frac{1}{N}S^{\frac{N}{2}}$ as $\mu\to0^+$, and
\begin{eqnarray*}
\|\nabla u_{a,\mu,-}\|_2^2,\|u_{a,\mu,-}\|_{2^*}^{2^*}\to S^{\frac{N}{2}}\quad\text{as }\mu\to0^+
\end{eqnarray*}
for $2<q<2+\frac{4}{N}$ (see also \cite[Theorem~1.7]{JL201}).  It follows that, up to a subsequence, $\{u_{a,\mu,-}\}$ is a minimizing sequence of the following minimizing problem:
\begin{eqnarray}\label{eq0028}
S=\inf_{u\in D^{1,2}(\bbr^N)\backslash\{0\}}\frac{\|\nabla u\|_2^2}{\|u\|_{2^*}^{2}}.
\end{eqnarray}
Since $N\geq5$, $U_\ve\in L^2(\bbr^N)$ for all $\ve>0$.  We then choose $\ve_0>0$ such that $U_{\ve_0}\in\mathcal{S}_a$.
By \cite[Lemma~4.2]{S20}, there exists $t(\mu)>0$ such that $(U_{\ve_0})_{t(\mu)}\in\mathcal{P}_-^{a,\mu}$ for $\mu>0$ sufficiently small, that is,
\begin{eqnarray*}
[t(\mu)]^2S^{\frac{N}{2}}=\mu\gamma_q\|U_{\ve_0}\|_q^q[t(\mu)]^{q\gamma_q}+[t(\mu)]^{2^*}S^{\frac{N}{2}}.
\end{eqnarray*}
Clearly, by the implicit function theorem, $t(\mu)$ is of class $C^1$ for $|\mu|<<1$ such that $t(0)=1$.  It follows from $S^{\frac{N}{2}}(1-[t(\mu)]^{2^*-2})=\mu\gamma_q\|U_{\ve_0}\|_q^q[t(\mu)]^{q\gamma_q-2}$ that
\begin{eqnarray}\label{eq0078}
t(\mu)=1-\frac{\gamma_q\|U_{\ve_0}\|_q^q}{(2^*-2)S^{\frac{N}{2}}}\mu+o(\mu),
\end{eqnarray}
which implies
\begin{eqnarray}
m_{a,\mu}^-&\leq&\mathcal{E}_\mu((U_{\ve_0})_{t(\mu)})\notag\\
&=&\frac{1}{N}S^{\frac{N}{2}}-\frac{\mu\gamma_q\|U_{\ve_0}\|_q^q}{2^*}-\frac{\mu}{q}(1-\frac{q\gamma_q}{2^*})\|U_{\ve_0}\|_q^q+o(\mu)\notag\\
&=&\frac{1}{N}S^{\frac{N}{2}}-\frac{\mu\|U_{\ve_0}\|_q^q}{q}+o(\mu)\label{eq0080}
\end{eqnarray}
for $N\geq5$.  Since we have $m_{a,\mu}^-=\frac{1}{N}\|\nabla u_{a,\mu,-}\|_2^2-\frac{\mu}{q}(1-\frac{q\gamma_q}{2^*})\|u_{a,\mu,-}\|_q^q$, by \eqref{eq0080},
\begin{eqnarray}\label{eq0030}
\frac{1}{N}\|\nabla u_{a,\mu,-}\|_2^2-\frac{\mu}{q}(1-\frac{q\gamma_q}{2^*})\|u_{a,\mu,-}\|_q^q\leq\frac{1}{N}S^{\frac{N}{2}}-\frac{\mu\|U_{\ve_0}\|_q^q}{q}+o(\mu).
\end{eqnarray}
On the other hand, by \eqref{eq0028} and $u_{a,\mu,-}\in\mathcal{P}_-^{a,\mu}$,
\begin{eqnarray*}
S&\leq&\frac{\|\nabla u_{a,\mu,-}\|_2^2}{\|u_{a,\mu,-}\|_{2^*}^{2}}\\
&=&\frac{\|\nabla u_{a,\mu,-}\|_2^2}{(\|\nabla u_{a,\mu,-}\|_2^2-\mu\gamma_q\|u_{a,\mu,-}\|_q^q)^{\frac{2}{2^*}}}\\
&=&(\|\nabla u_{a,\mu,-}\|_2^2-\mu\gamma_q\|u_{a,\mu,-}\|_q^q)^{\frac{2}{N}}+\frac{\mu\gamma_q\|u_{a,\mu,-}\|_q^q}{S^{\frac{N-2}{2}}}+o(\mu\|u_{a,\mu,-}\|_q^q).
\end{eqnarray*}
It follows that
\begin{eqnarray}\label{eq0031}
\|\nabla u_{a,\mu,-}\|_2^2\geq S^{\frac{N}{2}}-\frac{N-2}{2}\mu\gamma_q\|u_{a,\mu,-}\|_q^q+o(\mu\|u_{a,\mu,-}\|_q^q).
\end{eqnarray}
Combining \eqref{eq0030} and \eqref{eq0031}, we have
\begin{eqnarray}\label{eq0032}
\|u_{a,\mu,-}\|_q^q\geq\|U_{\ve_0}\|_q^q+o(1).
\end{eqnarray}
Since $\{u_{a,\mu,-}\}$ is bounded in $H^1(\bbr^N)$, $u_{a,\mu,-}\rightharpoonup u_{0,-}$ weakly in $H^1(\bbr^N)$ as $\mu\to0^+$ up to a subsequence.
Since $u_{a,\mu,-}$ is radial and decreasing for $r=|x|$,
\begin{eqnarray*}
\sup_{y\in\bbr^N}\int_{B_1(y)}|u_{a,\mu,-}|^2dx=\int_{B_1(0)}|u_{a,\mu,-}|^2dx.
\end{eqnarray*}
Thus, by \eqref{eq0032}, Lions' lemma \cite[Lemma~1.21]{W96} and the Sobolev embedding theorem, $u_{0,-}\not=0$.  Note that it is standard to show that $u_{0,-}$ is a weak solution of the following equation,
\begin{eqnarray*}
-\Delta U=U^{2^*-1},\quad\text{in }\bbr^N,
\end{eqnarray*}
thus, we must have $\|\nabla u_{0,-}\|_2^2\geq S^{\frac{N}{2}}$.  It follows from $\|\nabla u_{a,\mu,-}\|_2^2\to S^{\frac{N}{2}}$ as $\mu\to0^+$ that $u_{a,\mu,-}\to u_{0,-}$ strongly in $D^{1,2}(\bbr^N)$ as $\mu\to0^+$ up to a subsequence, which implies $u_{0,-}=U_{\ve}$ for some $\ve>0$.  Since $\|U_{\ve_0}\|_2^2=a^2$, by the Fatou lemma and \eqref{eq0032},
\begin{eqnarray*}
\|U_{\ve}\|_q^q\geq\|U_{\ve_0}\|_q^q\quad\text{and}\quad\|U_{\ve}\|_2\leq\|U_{\ve_0}\|_2^2.
\end{eqnarray*}
Hence, we must have $\ve=\ve_0$ and thus, $\|u_{0,-}\|_2=\|U_{\ve_0}\|_2^2=a^2$, which implies $u_{a,\mu,-}\to U_{\ve_0}$
strongly in $H^1(\bbr^N)$ as $\mu\to0^+$ up to a subsequence.  Since $U_{\ve_0}$ is the unique Aubin-Talanti babble in $\mathcal{S}_a$, we have $u_{a,\mu,-}\to U_{\ve_0}$
strongly in $H^1(\bbr^N)$ as $\mu\to0^+$.  Moreover, since $u_{a,\mu,-}$ is a solution of \eqref{eq0001}, by the Pohozaev identity and $u_{a,\mu,-}\in\mathcal{P}_{a,\mu}$, we have
\begin{eqnarray}\label{eq2238}
-\lambda_{a,\mu,-}\|u_{a,\mu,-}\|_2^2=(1-\gamma_q)\mu\|u_{a,\mu,-}\|_q^q,
\end{eqnarray}
which implies $\lambda_{a,\mu,-}\to0$ as $\mu\to0^+$.
It remains to prove that $u_{a,\mu,-}$ is the unique minimizer of $\mathcal{E}|_{\mathcal{S}_a}(u)$ on $\mathcal{P}_-^{a,\mu}$ for $\mu>0$ sufficiently small up to translations and rotations.
For this, let us first claim that
\begin{eqnarray}\label{eqnewnew6665}
u_{a,\mu,-}\lesssim\bigg(\frac{1}{1+r^2}\bigg)^{\frac{N-2}{2}}
\end{eqnarray}
for all $r\geq0$ in the case of $\mu>0$ sufficiently small.
Indeed, since $u_{a,\mu,-}$ is a positive and radially decreasing solution of \eqref{eq0001}, by Struss's radial lemma (cf. \cite[Lemma A.IV, Theorem A.I']{BL83} or \cite[Lemma~3.1]{MM14}), $u_{a,\mu,-}\lesssim r^{-\frac{N-1}{2}}$ for $r\geq1$.  Thus, by \eqref{eq2238}, $u_{a,\mu,-}$ satisfies
\begin{eqnarray}\label{eqnewnew6666}
-u_{a,\mu,-}''-\frac{N-1}{r}u_{a,\mu,-}'\lesssim  u_{a,\mu,-}^{\frac{4}{N-2}} u_{a,\mu, -} \lesssim r^{- (2+\delta)} u_{a,\mu, -}\quad\text{for }r\gtrsim1
\end{eqnarray}
for some $\delta >0$. By bootstrapping we obtain the desired decaying estimate (\ref{eqnewnew6665}).

Now, let us consider the following system:
\begin{eqnarray}\label{eq0125}
\left\{\aligned&\mathcal{F}(w,\alpha,\mu)=\Delta w-\alpha w+\mu w^{q-1}+w^{2^*-1},\\
&\mathcal{G}(w,\alpha,\mu)=\|w\|_2^2-a^2,
\endaligned\right.
\end{eqnarray}
where $\alpha,\mu>0$ are parameters.
It is easy to see that $\mathcal{F}(U_{\ve_0},0,0)=0$ and $\mathcal{G}(U_{\ve_0},0,0)=0$.  Let
\begin{eqnarray*}
\mathcal{L}(U_{\ve_0},0,0)=\left(\aligned\partial_w\mathcal{F}(U_{\ve_0},0,0)\quad\partial_\alpha\mathcal{F}(U_{\ve_0},0,0)\\
\partial_w\mathcal{G}(U_{\ve_0},0,0)\quad \partial_\alpha\mathcal{G}(U_{\ve_0},0,0)\endaligned\right)
\end{eqnarray*}
be the linearization of the system~\eqref{eq0125} at $(U_{\ve_0},0,0)$ in $H^1(\bbr^N)\times\bbr$, that is,
\begin{eqnarray*}
\partial_w\mathcal{F}(U_{\ve_0},0,0)=\Delta+(2^*-1)U_{\ve_0}^{2^*-2},\quad\partial_\alpha\mathcal{F}(U_{\ve_0},0,0)=-U_{\ve_0}
\end{eqnarray*}
and
\begin{eqnarray*}
\partial_w\mathcal{G}(U_{\ve_0},0,0)=2U_{\ve_0},\quad\partial_\alpha\mathcal{G}(U_{\ve_0},0,0)=0.
\end{eqnarray*}
Then $\mathcal{L}(U_{\ve_0},0,0)[(\phi,\tau)]=0$ if and only if
\begin{eqnarray}\label{eq0126}
\left\{\aligned&\Delta \phi+(2^*-1)U_{\ve_0}^{2^*-2}\phi-\tau U_{\ve_0}=0,\\
&\int_{\bbr^N}U_{\ve_0}\phi=0.
\endaligned\right.
\end{eqnarray}
We claim that in $H_{rad}^1(\bbr^N)\times\bbr$, $\mathcal{L}(U_{\ve_0},0,0)[(\phi,\tau)]=0$ if and only if $(\phi,\tau)=(0,0)$.  Let $\mathcal{L}(U_{\ve_0},0,0)[(\phi,\tau)]=0$ for some $(\phi,\tau)\in H_{rad}^1(\bbr^N)\times\bbr$.  Since it is well-known (cf. \cite{BE91}) that $W=\frac{N-2}{2}U_{\ve_0}+U_{\ve_0}'r$ is the unique radial solution of the following equation
\begin{eqnarray*}
\Delta \phi+(2^*-1)U_{\ve_0}^{2^*-2}\phi=0
\end{eqnarray*}
in $H^1_{rad}(\bbr^N)$, by multiplying the first equation of \eqref{eq0126} with $W$ and integrating by parts, we have
\begin{eqnarray*}
0=\tau\int_{\bbr^N}WU_{\ve_0}=-\tau\int_{\bbr^N}U_{\ve_0}^2.
\end{eqnarray*}
It follows that $\tau=0$ and thus, $\phi=CW$ for some constant $C\in\bbr$.  By the second equation of \eqref{eq0126},
\begin{eqnarray*}
0=\int_{\bbr^N}U_{\ve_0}\phi=-C\int_{\bbr^N}U_{\ve_0}^2,
\end{eqnarray*}
which implies that $\phi=0$.  Thus, the kernel of the linearization of the system~\eqref{eq0125} at $(U_{\ve_0},0,0)$ in $H_{rad}^1(\bbr^N)\times\bbr$ is trivial, which implies that the linear operator $\mathcal{L}(U_{\ve_0},0,0):H_{rad}^1(\bbr^N)\times\bbr\to H_{rad}^1(\bbr^N)\times\bbr$ is injective.  On the other hand, by similar arguments as used for Proposition~\ref{prop0004}, we know that all minimizers of $\mathcal{E}_\mu(u)|_{\mathcal{S}_a}$ on $\mathcal{P}_-^{a,\mu}$ are real valued, positive, radially symmetric and radially decreasing up to translations and rotations.  Now, suppose that there are at least two minimizers of $\mathcal{E}_\mu(u)|_{\mathcal{S}_a}$ on $\mathcal{P}_-^{a,\mu}$, say $u_{\mu}^*$ and $u_{\mu}^{**}$, then without loss of generality, we may assume that they are all real valued, positive, radially symmetric and radially decreasing.  The corresponding Lagrange multipliers are $\lambda_\mu^*$ and $\lambda_\mu^{**}$, respectively.  Let
\begin{eqnarray*}
w_\mu=\frac{u_{\mu}^*-u_{\mu}^{**}}{\|u_{\mu}^*-u_{\mu}^{**}\|_{H^1}+|\lambda_\mu^*-\lambda_\mu^{**}|}\quad\text{and}\quad \varsigma_\mu=\frac{\lambda_\mu^*-\lambda_\mu^{**}}{\|u_{\mu}^*-u_{\mu}^{**}\|_{H^1}+|\lambda_\mu^*-\lambda_\mu^{**}|},
\end{eqnarray*}
where $\|\cdot\|_{H^1}$ is the usual norm in $H^1(\bbr^N)$.
It is easy to see that $\{w_\mu\}$ is bounded in $H^{1}(\bbr^N)$ and $\{\varsigma_\mu\}$ is bounded.  Moreover, by \eqref{eq0001}, we also have
\begin{eqnarray}
-\Delta w_\mu-\lambda_\mu^*w_\mu-\varsigma_\mu u_\mu^{**}&=&\mu(q-1)\bigg(u_\mu^*+\theta_\mu(u_{\mu}^*-u_{\mu}^{**})\bigg)^{q-2}w_\mu\notag\\
&&+(2^*-1)\bigg(u_\mu^*+\theta_\mu'(u_{\mu}^*-u_{\mu}^{**})\bigg)^{2^*-2}w_\mu,\label{eq9999}
\end{eqnarray}
where $\theta_\mu,\theta_\mu'\in(0, 1)$.  Since $u_\mu^*$ and $u_\mu^{**}$ belong to $\mathcal{S}_a$, we also have
\begin{eqnarray*}
2\int_{\bbr^N}u_\mu^*w_\mu=-\|u_{\mu}^*-u_{\mu}^{**}\|_{2}\|w_\mu\|_{2}.
\end{eqnarray*}
Since the linear operator $\mathcal{L}(U_{\ve_0},0,0):H_{rad}^1(\bbr^N)\times\bbr\to H_{rad}^1(\bbr^N)\times\bbr$ is injective, it is standard to prove that
$(w_\mu, \varsigma_\mu)\rightharpoonup(0, 0)$ weakly in $D^{1,2}(\bbr^N)\times\bbr$ as $\mu\to0^+$.  Now, by multiplying \eqref{eq9999} with $w_\mu$ and integrating by parts,
we can use fact that $u_{\mu}^{**},u_{\mu}^{*}\to U_{\ve_0}$ strongly in $H^1(\bbr^N)$ as $\mu\to+\infty$ to show that
$(w_\mu, \varsigma_\mu)\to(0, 0)$ strongly in $D^{1,2}(\bbr^N)\times\bbr$ as $\mu\to0^+$.  Moreover, by \eqref{eq2238}, we also have
\begin{eqnarray*}
\varsigma_\mu\sim\mu\int_{\bbr^N}\bigg(u_\mu^*+\theta_\mu(u_{\mu}^*-u_{\mu}^{**})\bigg)^{q-1}w_\mu=o(\mu).
\end{eqnarray*}
By \eqref{eq2238}, \eqref{eqnewnew6665} and \eqref{eq9999},
\begin{eqnarray*}
-\Delta w_\mu-\frac12\lambda_\mu^*w_\mu\lesssim \frac{\mu}{r^{N-2}}\quad\text{for }r\gtrsim\frac{1}{|\lambda_\mu^*|^{\frac{1}{4}}}.
\end{eqnarray*}
By \eqref{eq2238}, we also have
\begin{eqnarray*}
-\Delta(r^{2-N})-\frac12\lambda_\mu^*r^{2-N}=-\frac12\lambda_\mu^*r^{2-N}\gtrsim\frac{\mu}{r^{N-2}}\quad\text{for }r\gtrsim\frac{1}{|\lambda_\mu^*|^{\frac{1}{4}}}.
\end{eqnarray*}
Since $w_\mu$ is radial and $\{w_\mu\}$ is bounded in $H^1(\bbr^N)$, by \cite[Lemma~A.2]{BL83},
\begin{eqnarray}\label{eq3338}
|w_\mu|\lesssim r^{-\frac{N-1}{2}}\quad\text{for } r\gtrsim1.
\end{eqnarray}
Thus, by the maximum principle,
\begin{eqnarray}\label{eq3438}
|w_\mu|\lesssim r^{2-N}\quad\text{for } r\gtrsim\frac{1}{|\lambda_\mu^*|^{\frac{1}{4}}}.
\end{eqnarray}
For $1\lesssim r\lesssim\frac{1}{|\lambda_\mu^*|^{\frac{1}{4}}}$, by \eqref{eqnewnew6665}, \eqref{eq9999} and \eqref{eq3338},
\begin{eqnarray*}
-\Delta w_\mu\lesssim \frac{1}{r^4}(\frac{1}{r^{\frac{N-1}{2}}}+\frac{1}{r^{N-2}})\lesssim r^{-\frac{7+N}{2}}\lesssim r^{-\frac{\alpha+N}{2}}
\end{eqnarray*}
in the case of $N\geq5$, where $\alpha=\frac{9}{2}$.  Recall that for $N\geq5$, $r^{2-\frac{\alpha+N}{2}}$ is also a superharmornic function.  Thus, by the maximum principle,
\begin{eqnarray}\label{eq4438}
|w_\mu|\lesssim r^{2-\frac{\alpha+N}{2}}\quad\text{for } 1\lesssim r\lesssim\frac{1}{|\lambda_\mu^*|^{\frac{1}{4}}}.
\end{eqnarray}
Note that by $w_\mu\to0$ strongly in $D^{1,2}(\bbr^N)$ as $\mu\to0^+$ and $\|w_\mu\|_{H^1}^2=1$, we know that $\|w_\mu\|_{2}^2=1+o_\mu(1)$.
Thus, by $w_\mu\to0$ strongly in $D^{1,2}(\bbr^N)$ as $\mu\to0^+$, the Sobolev embedding theorem and \eqref{eq3438} and \eqref{eq4438},
\begin{eqnarray*}
1\sim\int_{\bbr^N}|w_\mu|^2\lesssim o_\mu(1)+\int_{r_0}^{\frac{1}{|\lambda_\mu^*|^{\frac{1}{4}}}}r^{3-\alpha}+\int_{\frac{1}{|\lambda_\mu^*|^{\frac{1}{4}}}}^{+\infty}r^{3-N}=o_\mu(1)+\frac{1}{2}r_0^{-\frac{1}{2}},
\end{eqnarray*}
which is a contradiction by taking $r_0>0$ sufficiently large.  It follows that
$u_{a,\mu,-}$ is the unique minimizer of $\mathcal{E}|_{\mathcal{S}_a}(u)$ on $\mathcal{P}_-^{a,\mu}$ for $\mu>0$ sufficiently small up to translations and rotations if $N\geq5$.
\end{proof}

For $N=3,4$, The Aubin-Talanti babbles $U_\varepsilon\not\in L^2(\bbr^N)$.  Thus, we need to modify the arguments for Proposition~\ref{prop0006} to give a precise description of $u_{a,\mu,-}$ as $\mu\to0^+$ in these two cases.  As in the proof of Proposition~\ref{prop0006}, $\{u_{a,\mu,-}\}$ is also a minimizing sequence of the minimizing problem~\eqref{eq0028} in the cases $N=3,4$.  Since $u_{a,\mu,-}$ is radially symmetric for $r=|x|$, by Lions' result (cf. \cite[Theorem~1.41]{W96}), up to subsequence, there exists $\sigma_\mu>0$ such that for some $\ve_*>0$,
\begin{eqnarray}\label{eqnew6666}
v_{a,\mu,-}(x)=\sigma_\mu^{\frac{N-2}{2}}u_{a,\mu,-}(\sigma_\mu x)\to U_{\varepsilon_*}\text{ strongly in }D^{1,2}(\bbr^N)\text{ as }\mu\to0^+.
\end{eqnarray}
We also remark that since $U_{\varepsilon_*}\not\in L^2(\bbr^N)$ for $N=3,4$ and $\|v_{a,\mu,-}\|_2^2=\frac{a^2}{\sigma_\mu^2}$, by the Fatou lemma, we have $\sigma_\mu\to0$ as $\mu\to0^+$.
\begin{lemma}\label{lemn0001}
Let $N=3,4$ and $2<q<2^*$.
Then
\begin{eqnarray*}
1\sim\left\{\aligned &\frac{\mu\sigma_\mu^{N-\frac{N-2}{2}q}}{-\lambda_{a,\mu,-}},\quad \frac{N}{N-2}<q<2^*,\\
&\frac{\mu \sigma_\mu^{\frac{3}{2}}}{-\lambda_{a,\mu,-}}\ln\bigg(\frac{1}{\sqrt{-\lambda_{a,\mu,-}}\sigma_\mu}\bigg),\quad N=3, q=3,\\
&\frac{\mu \sigma_\mu^{3-\frac{q}{2}}\bigg(\sqrt{-\lambda_{a,\mu,-}}\sigma_\mu\bigg)^{q-3}}{-\lambda_{a,\mu,-}},\quad N=3, 2<q<3.\endaligned\right.
\end{eqnarray*}
\end{lemma}
\begin{proof}
By the equation~\eqref{eq0001}, we know that
$v_{a,\mu,-}$ satisfies
\begin{eqnarray}\label{eqnew2200}
-\Delta v_{a,\mu,-}-\lambda_{a,\mu,-}\sigma_{\mu}^2v_{a,\mu,-}=\mu\sigma_{\mu}^{N-\frac{N-2}{2}q}v_{a,\mu,-}^{q-1}+v_{a,\mu,-}^{2^*-1}\quad\text{in }\bbr^N.
\end{eqnarray}
It follows from \eqref{eq2238} that
\begin{eqnarray}\label{eq2239}
-\lambda_{a,\mu,-}\sigma_{\mu}^2\|v_{a,\mu,-}\|_2^2=(1-\gamma_q)\mu\sigma_{\mu}^{N-\frac{N-2}{2}q}\|v_{a,\mu,-}\|_q^q.
\end{eqnarray}
Recall that $\|u_{a,\mu,-}\|_2^2=\sigma_{\mu}^2\|v_{a,\mu,-}\|_2^2=a^2$ and $\|\nabla u_{a,\mu,-}\|_2^2\to S^{\frac{N}{2}}$ as $\mu\to0^+$, by \eqref{eq2238} and the H\"older inequality, $\lambda_{a,\mu,-}\to0$ as $\mu\to0^+$.  Clearly,
\begin{eqnarray}\label{eqnew7777}
\mu\sigma_{\mu}^{N-\frac{N-2}{2}q}\to0\quad\text{as }\mu\to0^+.
\end{eqnarray}
By the H\"older inequality once more,
\begin{eqnarray*}
|\lambda_{a,\mu,-}|\sigma_{\mu}^2\|v_{a,\mu,-}\|_2^2\lesssim\mu\sigma_{\mu}^{N-\frac{N-2}{2}q}\|v_{a,\mu,-}\|_2^{N-\frac{N-2}{2}q}.
\end{eqnarray*}
Since $q>2$, $N-\frac{N-2}{2}q<2$.  It follows from $\|v_{a,\mu,-}\|_2^2\sim\sigma_{\mu}^{-2}\to+\infty$ as $\mu\to0^+$ that
\begin{eqnarray}\label{eqnew7776}
|\lambda_{a,\mu,-}|\sigma_{\mu}^2=o(\mu\sigma_{\mu}^{N-\frac{N-2}{2}q})\quad\text{as }\mu\to0^+.
\end{eqnarray}
Recall that $v_{a,\mu,-}\to U_{\varepsilon_*}$ strongly in $D^{1,2}(\bbr^N)$ as $\mu\to0^+$ up to a subsequence, by \eqref{eqnew7777}-\eqref{eqnew7776}, adapting the Moser iteration (cf. \cite[B.3 Lemma]{S00}) and the $L^p$ theory of elliptic equations to \eqref{eqnew2200} and the Sobolev embedding theorem,
\begin{eqnarray}\label{eqnew1122}
v_{a,\mu,-}\to U_{\varepsilon_*}\quad\text{strongly in $L^\infty(\bbr^N)$ as $\mu\to0^+$ up to a subsequence}
\end{eqnarray}
In what follows, we follow the ideas in \cite{AP86} (see also \cite{GS03,KP89}) to drive a uniformly upper bound of $v_{a,\mu,-}$.  We define
\begin{eqnarray*}
\widetilde{v}_{a,\mu,-}=\frac{1}{v_{a,\mu,-}(0)}v_{a,\mu,-}(\sqrt{v_{a,\mu,-}(0)}x).
\end{eqnarray*}
Since $\widetilde{v}_{a,\mu,-}$ is radial, $\widetilde{v}_{a,\mu,-}$ satisfies
\begin{eqnarray}\label{eqnew2201}
-\widetilde{v}_{a,\mu,-}''-\frac{N-1}{r}\widetilde{v}_{a,\mu,-}'=f(\widetilde{v}_{a,\mu,-})\quad\text{in }\bbr^N.
\end{eqnarray}
where
\begin{eqnarray*}
f(\widetilde{v}_{a,\mu,-})&=&\lambda_{a,\mu,-}\sigma_{\mu}^2v_{a,\mu,-}(0)\widetilde{v}_{a,\mu,-}+\mu\sigma_{\mu}^{N-\frac{N-2}{2}q}[v_{a,\mu,-}(0)]^{q-1}\widetilde{v}_{a,\mu,-}^{q-1}\\
&&+[v_{a,\mu,-}(0)]^{2^*-1}\widetilde{v}_{a,\mu,-}^{2^*-1}.
\end{eqnarray*}
Let
\begin{eqnarray*}
H(r)=r^N(\widetilde{v}_{a,\mu,-}')^2+(N-2)r^{N-1}\widetilde{v}_{a,\mu,-}\widetilde{v}_{a,\mu,-}'+\frac{N-2}{N}r^N\widetilde{v}_{a,\mu,-}f(\widetilde{v}_{a,\mu,-}).
\end{eqnarray*}
Then by direct calculations and using \eqref{eqnew7776}-\eqref{eqnew2201},
\begin{eqnarray*}
H'(r)&=&\frac{r^N\widetilde{v}_{a,\mu,-}'}{N}(4|\lambda_{a,\mu,-}|\sigma_{\mu}^2-(N-2)(2^*-q)\mu\sigma_{\mu}^{N-\frac{N-2}{2}q}v_{a,\mu,-}^{q-2})v_{a,\mu,-}\\
&=&\mu\sigma_{\mu}^{N-\frac{N-2}{2}q}\frac{r^N\widetilde{v}_{a,\mu,-}'}{N}(o_\mu(1)-(N-2)(2^*-q)U_{\varepsilon_*}^{q-2})v_{a,\mu,-}.
\end{eqnarray*}
Since $v_{a,\mu,-}>0$, $\widetilde{v}_{a,\mu,-}'<0$ and $v_{a,\mu,-}$ exponentially decays to zero as $r\to+\infty$, there exists $r_\mu>0$, $H'(r)>0$ for $0<r<r_\mu$ and $H'(r)<0$ for $r>r_\mu$.
Thus, $H(r)>H(0)=0$ for all $r>0$.  Let
\begin{eqnarray*}
\Psi(r)=\frac{-\widetilde{v}_{a,\mu,-}'}{r\widetilde{v}_{a,\mu,-}^{\frac{N}{N-2}}}.
\end{eqnarray*}
Then by direct calculations and using \eqref{eqnew2201},
\begin{eqnarray*}
\Psi'(r)=\frac{N}{N-2}r^{-(1+N)}\widetilde{v}_{a,\mu,-}^{-\frac{2N-2}{N-2}}H(r)>0
\end{eqnarray*}
for all $r>0$.  It follows from \eqref{eqnew2201} once more that
\begin{eqnarray*}
\Psi(r)>\Psi(0)=-\widetilde{v}_{a,\mu,-}''(0)=\frac{d_\mu}{N}
\end{eqnarray*}
where
\begin{eqnarray*}
d_\mu=\lambda_{a,\mu,-}\sigma_{\mu}^2v_{a,\mu,-}(0)+\mu\sigma_{\mu}^{N-\frac{N-2}{2}q}[v_{a,\mu,-}(0)]^{q-1}+[v_{a,\mu,-}(0)]^{2^*-1}.
\end{eqnarray*}
Let
\begin{eqnarray*}
Z_\mu(r)=\frac{1}{(1+\frac{d_\mu}{N(N-2)} r^2)^{\frac{N-2}{2}}}.
\end{eqnarray*}
Then it is easy to check that $\frac{-Z_\mu'(r)}{[Z_\mu(r)]^{\frac{N}{N-2}}}=\frac{d_\mu}{N} r$.  It follows that
\begin{eqnarray*}
\frac{\widetilde{v}_{a,\mu,-}'}{\widetilde{v}_{a,\mu,-}^{\frac{N}{N-2}}}\leq\frac{Z_\mu'(r)}{[Z_\mu(r)]^{\frac{N}{N-2}}}\quad\text{for all }r>0,
\end{eqnarray*}
which together with \eqref{eqnew1122}, implies
\begin{eqnarray}\label{eqnew1111}
v_{a,\mu,-}\lesssim \frac{1}{(1+r^2)^{\frac{N-2}{2}}}\quad\text{for all }r>0
\end{eqnarray}
uniformly for $\mu>0$ sufficiently small.
Now, for the cases $\frac{N}{N-2}<q<2^*$,
\begin{eqnarray*}
\|v_{a,\mu,-}\|_q^q\lesssim\int_0^{+\infty}\frac{1}{(1+r^2)^{\frac{(N-2)q}{2}}}r^{N-1}dr\lesssim1.
\end{eqnarray*}
By the Fatou lemma, $\|v_{a,\mu,-}\|_q^q\geq\|U_{\ve_*}\|_q^q+o_\mu(1)\gtrsim1$.  Thus, by \eqref{eq2239},
\begin{eqnarray*}
\frac{\mu\sigma_\mu^{N-\frac{N-2}{2}q}}{-\lambda_{a,\mu,-}}\sim1\quad\text{for }\frac{N}{N-2}<q<2^*.
\end{eqnarray*}
For $N=3$ and $q=3$, we need to drive the uniformly exponential decay of $v_{a,\mu,-}$ at infinitely both from below and above to obtain the conclusions.  Let
\begin{eqnarray*}
\Phi=r^{-1}e^{-\sqrt{|\lambda_{a,\mu,-}|}\sigma_{\mu} r}.
\end{eqnarray*}
Then it is easy to check (cf. \cite{MM14}) that $-\Delta\Phi-\lambda_{a,\mu,-}\sigma_{\mu}^2\Phi\leq0$ for $r\geq1$ in the case of $N=3$.  Since $v_{a,\mu,-}\to U_{\varepsilon_*}$ strongly in $L^\infty(\bbr^N)$ as $\mu\to0^+$ up to a subsequence, by the maximum principle,
\begin{eqnarray}\label{eqnew1110}
v_{a,\mu,-}\gtrsim r^{-1}e^{-\sqrt{|\lambda_{a,\mu,-}|}\sigma_{\mu} r}\quad\text{for }r\geq1
\end{eqnarray}
in the case of $N=3$.
On the other hand, let
\begin{eqnarray*}
\Upsilon=r^{-1}e^{-\frac{1}{2}\sqrt{|\lambda_{a,\mu,-}|}\sigma_{\mu} r}.
\end{eqnarray*}
Then it is also easy to check that $-\Delta\Upsilon-\frac12\lambda_{a,\mu,-}\sigma_{\mu}^2\Upsilon\geq0$ for $r\geq1$.  Since $\mu\sigma_\mu^{3-\frac{q}{2}}\to0$ as $\mu\to0^+$, by \eqref{eqnew1111}, for
\begin{eqnarray*}
r\gtrsim\frac{1}{|\lambda_{a,\mu,-}|\sigma_{\mu}^2},
\end{eqnarray*}
we have $-\Delta v_{a,\mu,-}-\frac{1}{2}\lambda_{a,\mu,-}\sigma_{\mu}^2v_{a,\mu,-}\leq0$ in the case of $N=3$.  Thus, by the maximum principle and \eqref{eqnew1122} once more,
\begin{eqnarray}\label{eqnew1130}
v_{a,\mu,-}\lesssim r^{-1}e^{-\frac12\sqrt{|\lambda_{a,\mu,-}|}\sigma_{\mu} r}\quad\text{for }r\gtrsim\frac{1}{|\lambda_{a,\mu,-}|\sigma_{\mu}^2}.
\end{eqnarray}
For $q=3$ and $N=3$, by \eqref{eqnew1111} and \eqref{eqnew1130},
\begin{eqnarray*}
\|v_{a,\mu,-}\|_3^3&\lesssim&\int_1^{\frac{1}{|\lambda_{a,\mu,-}|\sigma_{\mu}^2}}r^{-1}
+\int_{\frac{1}{|\lambda_{a,\mu,-}|\sigma_{\mu}^2}}^{+\infty}e^{-\sqrt{|\lambda_{a,\mu,-}|}\sigma_{\mu} r}\\
&\lesssim&\ln\bigg(\frac{1}{\sqrt{|\lambda_{a,\mu,-}|}\sigma_{\mu}}\bigg).
\end{eqnarray*}
By \eqref{eqnew1110}, for $q=3$ and $N=3$, we also have
\begin{eqnarray*}
\|v_{a,\mu,-}\|_3^3\gtrsim\int_1^{\frac{1}{\sqrt{|\lambda_{a,\mu,-}|}\sigma_{\mu}}}r^{-1}\gtrsim\ln\bigg(\frac{1}{\sqrt{|\lambda_{a,\mu,-}|}\sigma_{\mu}}\bigg).
\end{eqnarray*}
Thus, the conclusion for $q=3$ and $N=3$ then follows \eqref{eq2239}.  For $N=3$ and $2<q<3$, we need to construct a newly upper bound of $v_{a,\mu,-}$ by adapting ideas in \cite{DPG13}.  By \eqref{eqnew1111}, we have
\begin{eqnarray*}
-\Delta v_{a,\mu,-}-\frac12\lambda_{a,\mu,-}\sigma_{\mu}^2v_{a,\mu,-}\lesssim\mu\sigma_{\mu}^{3-\frac{q}{2}}r^{1-q}\quad\text{for }r\gtrsim\frac{1}{\sqrt{|\lambda_{a,\mu,-}|}\sigma_{\mu}}.
\end{eqnarray*}
On the other hand, let $\phi_\mu\sim\frac{\mu\sigma_{\mu}^{3-\frac{q}{2}}}{|\lambda_{a,\mu,-}|\sigma_{\mu}^2}r^{1-q}$, then by direct calculations,
\begin{eqnarray*}
-\Delta \phi_\mu-\frac12\lambda_{a,\mu,-}\sigma_{\mu}^2\phi_\mu\gtrsim\mu\sigma_{\mu}^{3-\frac{q}{2}}r^{1-q}\quad\text{for }r\gtrsim\frac{1}{\sqrt{|\lambda_{a,\mu,-}|}\sigma_{\mu}}.
\end{eqnarray*}
By \eqref{eqnew1111} and the maximum principle,
\begin{eqnarray}\label{eqnew1119}
v_{a,\mu,-}\lesssim\frac{\mu\sigma_{\mu}^{3-\frac{q}{2}}}{|\lambda_{a,\mu,-}|\sigma_{\mu}^2}r^{1-q} \quad\text{for }r\gtrsim\frac{1}{\sqrt{|\lambda_{a,\mu,-}|}\sigma_{\mu}}.
\end{eqnarray}
Now, using \eqref{eqnew1119} as a new barrier and by \eqref{eqnew1111}, we know that
\begin{eqnarray*}
-\Delta v_{a,\mu,-}-\frac12\lambda_{a,\mu,-}\sigma_{\mu}^2v_{a,\mu,-}\lesssim\mu\sigma_{\mu}^{3-\frac{q}{2}}\bigg(\frac{\mu\sigma_{\mu}^{3-\frac{q}{2}}}{|\lambda_{a,\mu,-}|\sigma_{\mu}^2}r^{1-q}\bigg)^{q-1}
\end{eqnarray*}
for $r\gtrsim\frac{1}{\sqrt{|\lambda_{a,\mu,-}|}\sigma_{\mu}}$.  Thus, by similar comparisons, we have
\begin{eqnarray*}
v_{a,\mu,-}\lesssim\bigg(\frac{\mu\sigma_{\mu}^{3-\frac{q}{2}}}{|\lambda_{a,\mu,-}|\sigma_{\mu}^2}\bigg)^{q}r^{-(q-1)^2} \quad\text{for }r\gtrsim\frac{1}{\sqrt{|\lambda_{a,\mu,-}|}\sigma_{\mu}}.
\end{eqnarray*}
By iterating the above arguments $n$ times for a sufficiently large $n$ such that $q(q-1)^n-3>0$, we have
\begin{eqnarray}\label{eqn7777}
v_{a,\mu,-}\lesssim\bigg(\frac{\mu\sigma_{\mu}^{3-\frac{q}{2}}}{|\lambda_{a,\mu,-}|\sigma_{\mu}^2}\bigg)^{s_n}r^{-(q-1)^n} \quad\text{for }r\gtrsim\frac{1}{\sqrt{|\lambda_{a,\mu,-}|}\sigma_{\mu}},
\end{eqnarray}
where $s_n=s_{n-1}(q-1)+1$ which implies
\begin{eqnarray*}
s_n=\frac{(q-1)^{n+1}-1}{q-2}.
\end{eqnarray*}
By \eqref{eqnew1110}, we have
\begin{eqnarray}\label{eqn7776}
\|v_{a,\mu,-}\|_q^q\gtrsim\int_1^{\frac{1}{\sqrt{|\lambda_{a,\mu,-}|}\sigma_{\mu}}}r^{2-q}e^{-q\sqrt{|\lambda_{a,\mu,-}|}\sigma_{\mu} r}
\gtrsim\bigg(\frac{1}{|\lambda_{a,\mu,-}|\sigma_{\mu}^2}\bigg)^{\frac{3-q}{2}}
\end{eqnarray}
and
\begin{eqnarray*}
\|v_{a,\mu,-}\|_2^2\gtrsim\int_1^{\frac{1}{\sqrt{|\lambda_{a,\mu,-}|}\sigma_{\mu}}}e^{-2\sqrt{|\lambda_{a,\mu,-}|}\sigma_{\mu} r}
\gtrsim\bigg(\frac{1}{|\lambda_{a,\mu,-}|\sigma_{\mu}^2}\bigg)^{\frac{1}{2}}.
\end{eqnarray*}
It follows from \eqref{eq2239} that
\begin{eqnarray}\label{eqn7778}
|\lambda_{a,\mu,-}|\gtrsim\mu\sigma_\mu^{3-\frac{q}{2}}\bigg(\frac{1}{|\lambda_{a,\mu,-}|\sigma_{\mu}^2}\bigg)^{\frac{3-q}{2}}\quad\text{and}\quad
\bigg(\frac{1}{|\lambda_{a,\mu,-}|\sigma_{\mu}^2}\bigg)^{\frac{1}{2}}|\lambda_{a,\mu,-}|\lesssim\mu
\end{eqnarray}
which implies
\begin{eqnarray*}
\frac{\mu\sigma_{\mu}^{3-\frac{q}{2}}}{|\lambda_{a,\mu,-}|\sigma_{\mu}^2}\lesssim|\lambda_{a,\mu,-}|^{\frac{3-q}{2}}\sigma_{\mu}^{1-q}\lesssim\sigma_{\mu}^{2(2-q)}.
\end{eqnarray*}
Now, by \eqref{eqn7777} and \eqref{eqn7778},
\begin{eqnarray}
\|v_{a,\mu,-}\|_q^q&\lesssim&\int_1^{\frac{1}{\sqrt{|\lambda_{a,\mu,-}|}\sigma_{\mu}}}r^{2-q}+
(\sigma_{\mu})^{-2((q-1)^{n+1}-1)}\int_{\frac{1}{\sqrt{|\lambda_{a,\mu,-}|}\sigma_{\mu}}}^{+\infty}r^{2-q(q-1)^n}\notag\\
&\lesssim&\bigg(\frac{1}{|\lambda_{a,\mu,-}|\sigma_{\mu}^2}\bigg)^{\frac{3-q}{2}}
+\sigma_\mu^{2(q(q-1)^n-3)-2((q-1)^{n+1}-1)}\notag\\
&\lesssim&\bigg(\frac{1}{|\lambda_{a,\mu,-}|\sigma_{\mu}^2}\bigg)^{\frac{3-q}{2}}+\sigma_\mu^{(q-1)^n(2+o_n(1))}\notag\\
&=&\bigg(\frac{1}{|\lambda_{a,\mu,-}|\sigma_{\mu}^2}\bigg)^{\frac{3-q}{2}}.\label{eqn7775}
\end{eqnarray}
The conclusion for $N=3$ and $2<q<3$ follows from \eqref{eq2239}, \eqref{eqn7776} and \eqref{eqn7775}.
\end{proof}

With Lemma~\ref{lemn0001} in hands, we can obtain the following.
\begin{proposition}\label{prop0009}
Let $N=3,4$ and $2<q<2^*$.
Then
\begin{eqnarray*}
w_{\mu,-}=\ve_{\mu}^{\frac{N-2}{2}}u_{a,\mu,-}(\ve_{\mu} x)\to U_{\ve_*}\quad\text{strongly in }D^{1,2}(\bbr^N)\text{ as }\mu\to0^+
\end{eqnarray*}
up to a subsequence for some $\ve_*>0$,
where $\ve_{\mu}>0$ satisfies
\begin{eqnarray}\label{eq0071}
\mu\sim\left\{\aligned&\ve_{\mu}^{6-q}e^{-2\ve_{\mu}^{-2}},\quad N=4, 2<q<2^*,\\
&\ve_{\mu}^{\frac{q}{2}-1},\quad N=3, 3<q<2^*,\\
&\frac{\ve_{\mu}^{\frac{1}{2}}}{\ln(\frac{1}{\ve_\mu})},\quad N=3, q=3,\\
&\ve_\mu^{5-\frac{3q}{2}},\quad N=3, 2<q<3.\endaligned\right.
\end{eqnarray}
\end{proposition}
\begin{proof}
Let $\{V_{\ve}\}$ be the family given by \eqref{eq0069}.  Since $2\geq\frac{N}{N-2}$ for $N\geq4$, By \cite[(4.2)--(4.5)]{MM14},
\begin{eqnarray}\label{eq0037}
\|V_{\ve}\|_{q}^{q}\sim\left\{\aligned&\ve^{N-\frac{N-2}{2}q},\quad N=3,4, \frac{N}{N-2}<q<2^*,\\
&\ve^{\frac{3}{2}}\ln(R_\ve \ve^{-1}),\quad N=3, q=3,\\
&\ve^{3-\frac{q}{2}}(R_\ve \ve^{-1})^{3-q},\quad N=3, 2<q<3
\endaligned\right.
\end{eqnarray}
for $\ve>0$ sufficiently small.  By \cite[Lemmas~4.2, 5.1 and 6.1]{S20}, there exist $t_{\mu,\ve}>0$ such that
\begin{eqnarray*}
\|\nabla V_\ve\|_2^2=\mu\gamma_q\|V_\ve\|_q^qt_{\mu,\ve}^{q\gamma_q-2}+\|V_\ve\|_{2^*}^{2^*}t_{\mu,\ve}^{2^*-2}
\end{eqnarray*}
and
\begin{eqnarray*}
2\|\nabla V_\ve\|_2^2<\mu q\gamma_q^2\|V_\ve\|_q^qt_{\mu,\ve}^{q\gamma_q-2}+2^*\|V_\ve\|_{2^*}^{2^*}t_{\mu,\ve}^{2^*-2}.
\end{eqnarray*}
Thus, $\{t_{\mu,\ve}\}$ is uniformly bounded and bounded from below away from $0$ for all $\ve,\mu>0$ sufficiently small.  By \eqref{eq0023} and \eqref{eq0036},
\begin{eqnarray*}
S^{\frac{N}{2}}(1-t_{\mu,\ve_\mu}^{2^*-2})=\mu\gamma_q\|V_{\ve}\|_{q}^{q}t_{\mu,\ve}^{q\gamma_q-2}+O((R_\ve\ve^{-1})^{2-N})).
\end{eqnarray*}
Then we can use similar arguments as used for \eqref{eq0078} to show that
\begin{eqnarray*}
t_{\mu,\ve}=1-(1+o(1))\frac{\mu\gamma_q\|V_{\ve}\|_{q}^{q}+O((R_\ve\ve^{-1})^{2-N}))}{(2^*-2)S^{\frac{N}{2}}}
\end{eqnarray*}
and thus by  similar arguments as used for \eqref{eq0032},
\begin{eqnarray*}
\|u_{a,\mu,-}\|_q^q\geq(1+o(1))(\|V_{\ve}\|_{q}^{q}-C\mu^{-1}(R_\ve\ve^{-1})^{2-N}))),
\end{eqnarray*}
which together with \eqref{eq0023} and \eqref{eq0037}, implies
\begin{eqnarray*}
\|u_{a,\mu,-}\|_q^q\gtrsim\left\{\aligned&\ve^{4-q}-C\mu^{-1}e^{-2\ve^{-2}},\quad N=4, 2<q<2^*,\\
&\ve^{3-\frac{q}{2}}-C\mu^{-1}\ve^2,\quad N=3, 3<q<2^*,\\
&\ve^{\frac{3}{2}}\ln(\frac{1}{\ve})-C\mu^{-1}\ve^2,\quad N=3, q=3,\\
&\ve^{\frac{3q}{2}-3}-C\mu^{-1}\ve^2,\quad N=3, 2<q<3.
\endaligned\right.
\end{eqnarray*}
By choosing $\ve_\mu$ such that the right hand sides of the above estimate take the maximum, we have \eqref{eq0071} and
\begin{eqnarray}\label{eqnew1199}
\|u_{a,\mu,-}\|_q^q\gtrsim\left\{\aligned&\ve_\mu^{N-\frac{N-2}{2}q},\quad N=3,4, \frac{N}{N-2}<q<2^*,\\
&\ve_\mu^{\frac{3}{2}}\ln(\frac{1}{\ve_\mu}),\quad N=3, q=3,\\
&\ve_\mu^{\frac{3q}{2}-3},\quad N=3, 2<q<3.
\endaligned\right.
\end{eqnarray}
We define $w_{\mu,-}=\ve_{\mu}^{\frac{N-2}{2}}u_{a,\mu,-}(\ve_{\mu} x)$, then $\|w_{\mu,-}\|_{2^*}^{2^*}, \|\nabla w_{\mu,-}\|_{2}^{2}\sim1$ and by \eqref{eqnew1199},
\begin{eqnarray}\label{eqnew1299}
\|w_{\mu,-}\|_q^q\gtrsim\left\{\aligned&1,\quad N=3,4, \frac{N}{N-2}<q<2^*,\\
&\ln(\frac{1}{\ve_\mu}),\quad N=3, q=3,\\
&\ve_\mu^{2q-6},\quad N=3, 2<q<3.
\endaligned\right.
\end{eqnarray}
It is easy to see that
\begin{eqnarray}\label{eq8866}
\sigma_{\mu}^{N-\frac{N-2}{2}q}\|v_{a,\mu,-}\|_q^q=\|u_{a,\mu,-}\|_q^q=\ve_{\mu}^{N-\frac{N-2}{2}q}\|w_{\mu,-}\|_q^{q}.
\end{eqnarray}
Then by Lemma~\ref{lemn0001}, \eqref{eq2239} and \eqref{eqnew1299}, we have
\begin{eqnarray}\label{eq8877}
\sigma_{\mu}^{N-\frac{N-2}{2}q}\gtrsim\ve_{\mu}^{N-\frac{N-2}{2}q}
\end{eqnarray}
for $\frac{N}{N-2}<q<2^*$ and $N=3,4$.  On the other hand, we know that
\begin{eqnarray}\label{eqn1288}
w_{\mu,-}(x)=\bigg(\frac{\ve_\mu}{\sigma_\mu}\bigg)^{\frac{N-2}{2}}v_{a,\mu,-}\bigg(\frac{\ve_\mu}{\sigma_\mu}x\bigg)
\end{eqnarray}
and $\widetilde{w}_{\mu,-}$ satisfies
\begin{eqnarray*}
-\Delta \widetilde{w}_{\mu,-}=g(\widetilde{w}_{\mu,-})\quad\text{in }\bbr^N.
\end{eqnarray*}
where
\begin{eqnarray*}
\widetilde{w}_{\mu,-}=\frac{1}{w_{\mu,-}(0)}w_{\mu,-}([w_{\mu,-}(0)]^{s}x)
\end{eqnarray*}
with $s\in\bbr$ and
\begin{eqnarray*}
g(\widetilde{w}_{\mu,-})&=&\lambda_{a,\mu,-}\ve_{\mu}^2[w_{\mu,-}(0)]^{2s}\widetilde{w}_{\mu,-}+\mu\ve_{\mu}^{N-\frac{N-2}{2}q}[w_{\mu,-}(0)]^{2s+q-2}\widetilde{w}_{\mu,-}^{q-1}\\
&&+[w_{\mu,-}(0)]^{2s+2^*-2}\widetilde{w}_{\mu,-}^{2^*-1}.
\end{eqnarray*}
By similar arguments as used for \eqref{eqnew1111}, we have
\begin{eqnarray}\label{eqnew1131}
w_{\mu,-}\lesssim \frac{w_{\mu,-}(0)}{(1+b_\mu r^2)^{\frac{N-2}{2}}}\quad\text{for all }r>0,
\end{eqnarray}
where
\begin{eqnarray*}
b_\mu&=&[w_{\mu,-}(0)]^{2s-1}(\lambda_{a,\mu,-}\ve_{\mu}^2w_{\mu,-}(0)+\mu\ve_{\mu}^{N-\frac{N-2}{2}q}[w_{\mu,-}(0)]^{q-1}\\
&&+[w_{\mu,-}(0)]^{2^*-1}).
\end{eqnarray*}
We recall that $\mu,\sigma_\mu,\lambda_{a,\mu,-}\to0$ as $\mu\to0^+$, and by \eqref{eq2238}, we have $\lambda_{a,\mu,-}\lesssim\mu$.  Thus, by Lemma~\ref{lemn0001}, \eqref{eqnew1122} and \eqref{eqn1288},
\begin{eqnarray}\label{eq8855}
b_\mu&\sim&\bigg(\frac{\ve_\mu}{\sigma_\mu}\bigg)^{(N-2)s+2}.
\end{eqnarray}
Now, take $s=-1$ and by \eqref{eqnew1131}, we can use similar arguments in the proof of Lemma~\ref{lemn0001} to show that
\begin{eqnarray*}
\|w_{\mu,-}\|_q^q\lesssim\bigg(\frac{\ve_\mu}{\sigma_\mu}\bigg)^{\frac{N-2}{2}q-\frac{N}{2}(4-N)}
\end{eqnarray*}
for $\frac{N}{N-2}<q<2^*$ and $N=3,4$, which together with \eqref{eq8866}, implies that $\sigma_\mu\lesssim\ve_\mu$ for $\frac{N}{N-2}<q<2^*$ and $N=3,4$.
It follows from \eqref{eq8877} that $\sigma_\mu \sim\ve_\mu$ for $\frac{N}{N-2}<q<2^*$ and $N=3,4$.  For the case $N=3$ and $q=3$, by $\|w_{\mu,-}\|_2^2\sim\ve_{\mu}^{-2}$, Struss's radial lemma (cf. \cite[Lemma A.IV, Theorem A.I']{BL83} or \cite[Lemma~3.1]{MM14}) and similar arguments as used for \eqref{eqnew1130},
\begin{eqnarray}\label{eqnew1132}
w_{\mu,-}\lesssim \ve_{\mu}^{-2}r^{-1}e^{-\frac12\sqrt{|\lambda_{a,\mu,-}|}\ve_{\mu} r}\quad\text{for }r\gtrsim\frac{1}{\sqrt{|\lambda_{a,\mu,-}|}\ve_{\mu}}.
\end{eqnarray}
It follows from \eqref{eqnew1131} and \eqref{eq8855} that
\begin{eqnarray*}
\|w_{\mu,-}\|_q^q\lesssim\bigg(\frac{\ve_\mu}{\sigma_\mu}\bigg)^{\frac{3}{2}-\frac{3}{2}(s+2)}\ln(\frac{1}{\sqrt{|\lambda_{a,\mu,-}|}\sigma_\mu}).
\end{eqnarray*}
By Lemma~\ref{lemn0001}, taking $s=-1$ and \eqref{eq8866}, we have $\ve_\mu\gtrsim\sigma_\mu$.  By Lemma~\ref{lemn0001}, taking $s=2$ and \eqref{eq8866}, we have $\ve_\mu\lesssim\sigma_\mu$.  Thus, for $N=3$ and $q=3$, we also have $\ve_\mu\sim\sigma_\mu$.  For the case $N=3$ and $2<q<3$, by \eqref{eq2238}, \eqref{eqnew1199}, Lemma~\ref{lemn0001} and $\mu\sim\ve_\mu^{5-\frac{3q}{2}}$,
\begin{eqnarray*}
\sigma_\mu^{\frac{q}{5-q}}\gtrsim\ve_\mu^{\frac{q}{5-q}}\bigg(\frac{\mu}{\ve_\mu^{5-\frac{3q}{2}}}\bigg)^{\frac{3-q}{5-q}}\sim\ve_\mu^{\frac{q}{5-q}}
\end{eqnarray*}
which implies $\sigma_\mu\gtrsim\ve_\mu$.  Thus, by \eqref{eqnew1122}, we can adapt the maximum principle as that in the proof of Proposition~\ref{prop0009} to show that
\begin{eqnarray}\label{eqnewA1110}
w_{a,\mu,-}\gtrsim r^{-1}e^{-\sqrt{|\lambda_{a,\mu,-}|}\ve_{\mu} r}\quad\text{for }r\geq1.
\end{eqnarray}
By \eqref{eqnewA1110}, we can see that the estimates for \eqref{eqn7778} works for $\ve_\mu$ and thus, we have
\begin{eqnarray*}
|\lambda_{a,\mu,-}|\gtrsim\mu\ve_\mu^{3-\frac{q}{2}}\bigg(\frac{1}{|\lambda_{a,\mu,-}|\ve_{\mu}^2}\bigg)^{\frac{3-q}{2}},\quad
\bigg(\frac{1}{|\lambda_{a,\mu,-}|\ve_{\mu}^2}\bigg)^{\frac{1}{2}}|\lambda_{a,\mu,-}|\lesssim\mu
\end{eqnarray*}
and
\begin{eqnarray*}
\frac{\mu\ve_{\mu}^{3-\frac{q}{2}}}{|\lambda_{a,\mu,-}|\ve_{\mu}^2}\lesssim|\lambda_{a,\mu,-}|^{\frac{3-q}{2}}\ve_{\mu}^{1-q}\lesssim\ve_{\mu}^{2(2-q)}.
\end{eqnarray*}
Now, we can follow similar arguments as used in the proof of Lemma~\ref{lemn0001} to show that
\begin{eqnarray*}
\|w_{\mu,-}\|_q^q&\lesssim&\bigg(\frac{1}{|\lambda_{a,\mu,-}|\ve_{\mu}^2}\bigg)^{\frac{3-q}{2}},
\end{eqnarray*}
which, together with Lemma~\ref{lemn0001} and \eqref{eq8866}, implies that $\sigma_\mu\lesssim\ve_\mu$.  Thus, we also have $\sigma_\mu\sim\ve_\mu$ as $\mu\to0^+$ in the case of $N=3$ and $2<q<3$.
\end{proof}

We are ready to give the proofs of Theorem~\ref{thm0001} and \ref{thm0003}.

\vskip0.12in

\noindent\textbf{Proof of Theorem~\ref{thm0001}:}\quad It follows immediately from Lemma~\ref{lemma0002}, Propositions~\ref{prop0001} and \ref{prop0002}.
\hfill$\Box$

\vskip0.12in

\noindent\textbf{Proof of Theorem~\ref{thm0003}:}\quad It follows immediately from Propositions~\ref{prop0004}, \ref{prop0006} and \ref{prop0009}.
\hfill$\Box$

We close this section by

\vskip0.12in

\noindent\textbf{Proof of Theorem~\ref{thm0002}:}\quad
$(1)$\quad Since the proof is similar to that of Proposition~\ref{prop0004}, we only sketch it.  In the case of $q=2+\frac{4}{N}$, we have $\|\varphi\|_2^2=\|\phi_0\|_2^2$ for all minimizers of the Gagliardo--Nirenberg inequality~\eqref{eq0059}, where $\phi$ is the unique solution of \eqref{eqnew1010}.  Thus, we choose $\varphi=\frac{a}{\|\phi_0\|_2}\phi\in\mathcal{S}_a$ as a test function of $m_{a,\mu}^-$.  By using similar arguments as used in the proof of Proposition~\ref{prop0004} and direct calculations,
\begin{eqnarray*}
m_{a,\mu}^-\leq\frac{1}{N}(1-\frac{\mu}{\alpha_{N,q,a}})^{\frac{2^*}{2^*-2}}\bigg(\frac{\|\nabla \phi_0\|_2}{\|\phi_0\|_{2^*}}\bigg)^{N}.
\end{eqnarray*}
It follows from $u_{a,\mu,-}\in\mathcal{P}_{a,\mu}$, the Gagliardo--Nirenberg and Sobolev inequalities that
\begin{eqnarray}\label{eqnew7555}
S^{\frac{N}{2}}\leq\frac{\|\nabla u_{a,\mu,-}\|_2^2}{(1-\frac{\mu}{\alpha_{N,q,a}})^{\frac{2}{2^*-2}}}\leq\bigg(\frac{\|\nabla \phi_0\|_2}{\|\phi_0\|_{2^*}}\bigg)^{N},
\end{eqnarray}
which, together with $u_{a,\mu,-}\in\mathcal{P}_{a,\mu}$ once more and the Pohozaev identity satisfied by $u_{a,\mu,-}$, implies that
\begin{eqnarray*}
-\lambda_{\mu,-}=\frac{1-\gamma_q}{a^2}\mu\|u_{a,\mu,-}\|_q^q\geq(1+o_\mu(1))\frac{1-\gamma_q}{a^2\gamma_q}S^{\frac{N}{2}}(1-\frac{\mu}{\alpha_{N,q,a}})^{\frac{2}{2^*-2}}
\end{eqnarray*}
and
\begin{eqnarray*}
-\lambda_{\mu,-}=\frac{1-\gamma_q}{a^2}\mu\|u_{a,\mu,-}\|_q^q\leq(1-\frac{\mu}{\alpha_{N,q,a}})^{\frac{2}{2^*-2}}\bigg(\frac{\|\nabla \phi_0\|_2}{\|\phi_0\|_{2^*}}\bigg)^{N}.
\end{eqnarray*}
Thus, $\{v_{a,\mu,-}\}$ is bounded in $H^1(\bbr^N)$, where
\begin{eqnarray*}
v_{a,\mu,-}=(\frac{a}{\|\phi_0\|_2})^{\frac{N-2}{2}}s_\mu^{\frac{N}{2}}u_{a,\mu,-}(\frac{a}{\|\phi_0\|_2}s_\mu x)
\end{eqnarray*}
and $s_\mu=(1-\frac{\mu}{\alpha_{N,q,a}})^{-\frac{N-2}{4}}$.  Clearly, $v_{a,\mu,-}$ satisfies
\begin{eqnarray*}
-\Delta v_{a,\mu,-}=\lambda_{\mu,-}\frac{a^2}{\|\phi_0\|_2^2} s_\mu^{2} v_{a,\mu,-}+\mu(\frac{a}{\|\phi_0\|_2})^{\frac{4}{N}}v_{a,\mu,-}^{q-1}+s_\mu^{2-\frac{N}{2}(2^*-2)}v_{a,\mu,-}^{2^*-1}
\end{eqnarray*}
By \eqref{eq1101} and \cite[(I.3)]{W83}, we know that $\alpha_{N,q,a}(\frac{a}{\|\phi_0\|_2})^{\frac{4}{N}}=1$ for $q=2+\frac{4}{N}$.  On the other hand, since $v_{a,\mu,-}$ is radial, it is standard to show that $v_{a,\mu,-}\to \psi_{\nu_a',1}$ strongly in $H^1(\bbr^N)$ as $\mu\to \alpha_{N,q,a}^-$ up to a subsequence for some $\nu_a'>0$.

$(2)$\quad In the cases of $2+\frac{4}{N}<q<2^*$, $\frac{2}{q-2}-\frac{N}{2}\not=0$.  Thus, we can choose $\nu_a>0$, as that in \eqref{eq0010}, such that $\|\psi_{\nu_a,1}\|_2^2=a^2$.  Again, we use $\psi_{\nu_a,1}\in\mathcal{S}_a$ as a test function of $m_{a,\mu}^-$.  By using similar arguments as used in the proof of Proposition~\ref{prop0004} and direct calculations, $m_{a,\mu}^-\lesssim\mu^{-\frac{2}{q\gamma_q-2}}$ as $\mu\to+\infty$.  It follows that $u_{a,\mu,-}\to0$ strongly in $D^{1,2}(\bbr^N)\cap L^q(\bbr^N)$ as $\mu\to +\infty$.  This, together with $u_{a,\mu,-}\in\mathcal{P}_{a,\mu}$ and the Gagliardo--Nirenberg and Sobolev inequalities, implies
\begin{eqnarray*}
\|\nabla u_{a,\mu,-}\|_2^2\geq(1+o_\mu(1))(\mu\gamma_q a^{q-q\gamma_q}C_{N,q}^q)^{-\frac{2}{q\gamma_q-2}}.
\end{eqnarray*}
On the other hand, for the test function $\psi_{\nu_a,1}$, it satisfies
\begin{eqnarray*}
\|\nabla \psi_{\nu_a,1}\|_2^2=\mu\gamma_q\|\psi_{\nu_a,1}\|_q^qt_\mu^{q\gamma_q-2}+\|\psi_{\nu_a,1}\|_{2^*}^{2^*}t_\mu^{2^*-2},
\end{eqnarray*}
where $(\psi_{\nu_a,1})_{t_\mu}\in\mathcal{P}_{a,\mu}$.  It follows that
\begin{eqnarray*}
t_\mu\|\nabla \psi_{\nu_a,1}\|_2\leq\bigg(\frac{1}{\mu a^{q-q\gamma_q}\gamma_q C_{N,q}^q}\bigg)^{\frac{1}{q\gamma_q-2}}.
\end{eqnarray*}
Thus,
\begin{eqnarray*}
\mathcal{E}_\mu((\psi_{\nu_a,1})_{t_\mu})&=&(1+o_\mu(1))(\frac{1}{2}-\frac{1}{q\gamma_q})\|\nabla (\psi_{\nu_a,1})_{t_\mu}\|_2^2\\
&\leq&(1+o_\mu(1))(\frac{1}{2}-\frac{1}{q\gamma_q})\bigg(\frac{1}{\mu a^{q-q\gamma_q}\gamma_q C_{N,q}^q}\bigg)^{\frac{2}{q\gamma_q-2}}.
\end{eqnarray*}
Note that $\mathcal{E}_\mu((\psi_{\nu_a,1})_{t_\mu})\geq m_{a,\mu}^-$ and
\begin{eqnarray*}
m_{a,\mu}^-=\mathcal{E}_\mu(u_{a,\mu,-})=(1+o_\mu(1))(\frac{1}{2}-\frac{1}{q\gamma_q})\|\nabla u_{a,\mu,-}\|_2^2
\end{eqnarray*}
as $\mu\to+\infty$, we must have
\begin{eqnarray}\label{eq1110}
\|\nabla u_{a,\mu,-}\|_2^2=(1+o_\mu(1))(\mu\gamma_q a^{q-q\gamma_q}C_{N,q}^q)^{-\frac{2}{q\gamma_q-2}}.
\end{eqnarray}
As in $(1)$, $\{v_{a,\mu,-}\}$ is bounded in $H^1(\bbr^N)$, where $v_{a,\mu,-}=s_\mu^{\frac{N}{2}}u_{a,\mu,-}(s_\mu x)$ and $s_\mu=\mu^{\frac{1}{q\gamma_q-2}}$.  Again, $v_{a,\mu,-}$ satisfies
\begin{eqnarray*}
-\Delta v_{a,\mu,-}=\lambda_{\mu,-} s_\mu^{2} v_{a,\mu,-}+v_{a,\mu,-}^{q-1}+s_\mu^{2-\frac{N}{2}(2^*-2)}v_{a,\mu,-}^{2^*-1}.
\end{eqnarray*}
Using \eqref{eq1110}, the Pohozaev identity satisfied by $u_{a,\mu,-}$ and $u_{a,\mu,-}\in\mathcal{P}_{a,\mu}$ once more, we have
\begin{eqnarray*}
-\lambda_{\mu,-}=(1+o_\mu(1))\frac{1-\gamma_q}{a^2}(\mu\gamma_q a^{q-q\gamma_q}C_{N,q}^q)^{-\frac{2}{q\gamma_q-2}}.
\end{eqnarray*}
Now, by similar arguments as used in $(1)$, $v_{a,\mu,-}\to \psi_{\nu_a',1}$ strongly in $H^1(\bbr^N)$ as $\mu\to +\infty$ up to a subsequence for some $\nu_a'>0$.  Since in the cases of $2+\frac{4}{N}<q<2^*$, $\frac{2}{q-2}-\frac{N}{2}\not=0$.  By $\|v_{a,\mu,-}\|_2^2=a^2$, we must have $\nu_a'=\nu_a$.  By the uniqueness of $\psi_{\nu_a,1}$ in $\mathcal{S}_a$, we know that $v_{a,\mu,-}\to \psi_{\nu_a,1}$ strongly in $H^1(\bbr^N)$ as $\mu\to +\infty$.  Using the uniqueness of $\psi_{\nu_a,1}$ in $\mathcal{S}_a$ and the nondegenerate of $\psi_{\nu_a,1}$, we can prove the local uniqueness of $u_{a,\mu,-}$ for $\mu>0$ sufficiently large by adapting similar arguments as used for $u_{a,\mu,+}$ in the proof of Proposition~\ref{prop0004}.
\hfill$\Box$

\section{Acknowledgements}
The research of J. Wei is
partially supported by NSERC of Canada and the research of Y. Wu is supported by NSFC (No. 11701554, No. 11771319, No. 11971339).

\end{document}